\documentclass[a4paper,11pt,reqno,noindent]{amsart}
\usepackage[centertags]{amsmath}
\usepackage{amsfonts,amssymb,amsthm,amscd,dsfont,cases,amscd,esint,enumerate,array}
\usepackage[T1]{fontenc}
\usepackage{csquotes}
\usepackage[english]{babel}
\usepackage[applemac]{inputenc}
\usepackage{newlfont}
\usepackage{color}
\usepackage[body={15cm,21.5cm},centering]{geometry} 
\usepackage{fancyhdr}
\usepackage{enumerate}

\usepackage[colorlinks,citecolor=black,urlcolor=black]{hyperref}

\theoremstyle{plain}
\newtheorem{theor0}{Theorem}[section]
\newenvironment{theor}
  {\pushQED{\qed}\begin{theor0}}{\popQED\end{theor0}}
\newtheorem{lem0}[theor0]{Lemma}

\newtheorem{prop0}[theor0]{Proposition}
\newenvironment{prop}
  {\pushQED{\qed}\begin{prop0}}{\popQED\end{prop0}}
\newtheorem{cor0}[theor0]{Corollary}

\newtheorem{propr0}[theor0]{Property}

\newtheorem{hyp0}[theor0]{Hypothesis}

\newtheorem{result0}[theor0]{Result}

\newtheorem{conj0}[theor0]{Conjecture}

\newtheorem{heur0}[theor0]{Heuristics}

\theoremstyle{definition}
\newtheorem{defin0}[theor0]{Definition}
\newenvironment{defin}
  {\pushQED{\qed}\begin{defin0}}{\popQED\end{defin0}}
\newtheorem{rems0}[theor0]{Remarks}

\newtheorem{ex0}[theor0]{Example}

\newtheorem{exs0}[theor0]{Examples}

\newtheorem{rem0}[theor0]{Remark}
\newenvironment{rem}
  {\pushQED{\qed}\begin{rem0}}{\popQED\end{rem0}}
\newtheorem{qu0}[theor0]{Question}

\newtheorem{qus0}[theor0]{Questions}

  \newtheorem{as0}[theor0]{Assumption}

\newcommand{\N}{\mathbb N}
\newcommand{\e}{\varepsilon}

\newcommand{\R}{\mathbb R}
\newcommand{\Z}{\mathbb Z}

\newcommand{\B}{\mathcal B}
\newcommand{\X}{\mathcal X}
\newcommand{\Y}{\mathcal Y}
\newcommand{\Cc}{\mathcal C}

\newcommand{\Sp}{\mathbb S}
\newcommand{\Pc}{\mathcal P}
\newcommand{\Dc}{\mathcal D}

\newcommand{\F}{\mathcal F}

\newcommand{\M}{\mathcal M}

\newcommand{\A}{\mathcal A}
\newcommand{\p}{\mathbb{P}}
\newcommand{\E}{\mathbb{E}}

\newcommand{\Var}{\operatorname{Var}}
\newcommand{\Ent}{\operatorname{Ent}}
\newcommand{\osc}{\partial^{\mathrm{osc}}}
\newcommand{\oscd}[1]{\mathrel{\osc_{#1}}}
\newcommand{\supess}{\operatorname{sup\,ess}}
\newcommand{\supessd}[1]{\mathrel{\mathop{\supess}\limits_{#1}}}
\newcommand{\infessd}[1]{\mathrel{\mathop{\infess}\limits_{#1}}}
\newcommand{\infess}{\operatorname{inf\,ess}}
\newcommand{\Mes}{\operatorname{Mes}}
\newcommand{\loc}{{\operatorname{loc}}}
\newcommand{\Ld}{\operatorname{L}}
\newcommand{\supp}{\operatorname{supp}}
\newcommand{\diam}{\operatorname{diam}}

\newcommand{\var}[1]{\mathrm{Var}\left[#1\right]}
\newcommand{\varR}[1]{\mathrm{Var}_R\left[#1\right]}
\newcommand{\expecR}[1]{\mathbb{E}_R\left[ #1 \right]}

\newcommand{\ent}[1]{\mathrm{Ent}\!\left[#1\right]}

\newcommand{\expecM}[1]{\mathbb{E}\bigg[ #1 \bigg]}
\newcommand{\cov}[2]{\mathrm{Cov}\left[#1;#2\right]}
\newcommand{\expec}[1]{\mathbb{E}\left[ #1 \right]}
\newcommand{\expecp}[1]{\mathbb{E}'\left[ #1 \right]}
\newcommand{\pr}[1]{\mathbb{P}\left[ #1 \right]}
\newcommand{\prm}[1]{\mathbb{P}\big[ #1 \big]}
\newcommand{\expeC}[2]{\mathbb{E}\left[\left. #1 \,\right\|\,#2\right]}
\newcommand{\expeCm}[2]{\mathbb{E}\big[ #1 \,\big\|\,#2\big]}
\newcommand{\expeCM}[2]{\mathbb{E}\bigg[ #1 \,\bigg\|\,#2\bigg]}
\newcommand{\expeCpm}[2]{\mathbb{E}'\big[ #1 \,\big\|\,#2\big]}

\newcommand{\prC}[2]{\mathbb{P}\left[\left. #1 \,\right\|\,#2\right]}

\newcommand{\parfct}[1]{\partial_{#1}^{\operatorname{fct}}}
\newcommand{\parG}[1]{\partial_{#1}^{\operatorname{G}}}

\newcommand{\parsup}[1]{\partial_{#1}^{\operatorname{\sup}}}

\newcommand{\step}[1]{\noindent \textit{Step} #1.}

\numberwithin{equation}{section}
\newcolumntype{M}[1]{>{\centering\arraybackslash}m{#1}}

\title[Multiscale functional inequalities: Constructive approach]{Multiscale functional inequalities in probability:\\ Constructive approach}
\author[M. Duerinckx]{Mitia Duerinckx}
\author[A. Gloria]{Antoine Gloria}
\address[Mitia Duerinckx]{Laboratoire de Mathématique d'Orsay, UMR 8628, Université Paris-Sud, F-91405 Orsay, France \& Universit\'e Libre de Bruxelles, Département de Mathématique, Brussels, Belgium}
\email{mduerinc@ulb.ac.be}
\address[Antoine Gloria]{Sorbonne Universit\'e, CNRS, Universit\'e de Paris, Laboratoire Jacques-Louis Lions (LJLL), F-75005 Paris, France \& Universit\'e Libre de Bruxelles, Département de Mathématique, Brussels, Belgium}
\email{gloria@ljll.math.upmc.fr}

\begin{document}
\maketitle

\begin{abstract}
Consider an ergodic stationary random field $A$ on the ambient space $\R^d$.
In order to establish concentration properties for nonlinear functions $Z(A)$, it is standard 
to appeal to functional inequalities like Poincaré or logarithmic Sobolev inequalities in the probability space.
These inequalities are however only known to hold for a restricted class of laws (product measures, Gaussian measures with integrable covariance, or more general Gibbs measures with nicely behaved Hamiltonians).
In this contribution, we introduce variants of these inequalities, which we refer to as \emph{multiscale functional inequalities} and which still imply fine concentration properties,
and we develop a constructive approach to such inequalities.
We consider random fields that can be viewed as transformations of a product structure, for which the question is reduced to devising approximate chain rules for nonlinear random changes of variables.
This approach allows us to cover most examples of random fields arising in the modelling of heterogeneous materials in the applied sciences, including Gaussian fields with arbitrary covariance function, Poisson random inclusions with (unbounded)  random  radii, random parking and Mat\'ern-type processes, as well as Poisson random tessellations. The obtained multiscale functional inequalities, which we primarily develop here in view of their application to concentration and to quantitative stochastic homogenization, are of independent interest.
\end{abstract}

\section{Introduction}

This contribution focuses on functional inequalities in the probability space and constitutes the first and main part of a series of three articles (together with~\cite{DG1,DG3}) where we introduce  \emph{multiscale functional inequalities}, which are multiscale weighted versions of standard functional inequalities (Poincaré, covariance, and logarithmic Sobolev inequalities).
One of the main achievements of the present contribution is the proof that
most examples of random fields arising in the modelling of heterogeneous materials in the applied sciences, including
some important examples from stochastic geometry (the random parking process and Poisson random tessellations), do satisfy such multiscale functional inequalities whereas they do not satisfy their standard versions.
As shown in the companion article~\cite{DG1}, these weaker inequalities still imply fine concentration properties
and they can be used as convenient quantitative mixing assumptions in stochastic homogenization, which was our original motivation for this work (see Section~\ref{sec:homog} below for details).

\subsection{Multiscale functional inequalities}

Let $a=(a_x)_{x\in \Z^d}$ be a family of random variables on a probability space $(\Omega,\A,\p)$ and consider a $\sigma(a)$-measurable random variable~$Z(a)$.
If $a$ is a stationary Gaussian field on $\Z^d$ with integrable covariance function, the variance of $Z(a)$ is known to be controlled via the Poincaré inequality
\begin{equation}\label{0-int-0}
\var{Z(a)} \,\le \,C\,\E\bigg[{\sum_{x\in \Z^d} |\parfct{a_x} Z(a)|^2}\bigg],
\end{equation}
where $C>0$ only depends on the covariance function of $a$ and where $\parfct{a_x} Z(a)$ stands for the partial derivative  $\frac{\partial Z}{\partial a_x}(a)$ 
of $Z$ wrt the variable $a_x$.
Likewise, if $(a_x)_{x\in \Z^d}$ are independent and identically distributed (i.i.d.) random variables (non-necessarily Gaussian),  the variance of $Z(a)$ is controlled via the Poincaré inequality
\begin{equation}\label{0-int-1}
\var{Z(a)} \,\le \,C\,\E\bigg[{ \sum_{x\in \Z^d} |\osc_{a_x} Z(a)|^2}\bigg],
\end{equation}
where $\osc_{a_x} Z(a)$ now stands for the oscillation $\sup_{a_x} Z(a)-\inf_{a_x} Z(a)$ of $Z$ wrt the variable~$a_x$.
Functional inequalities like~\eqref{0-int-0} or~\eqref{0-int-1} are known to imply fine concentration properties for random variables $Z(a)$ and have been extensively used in mathematical physics, for instance in the context of phase transitions for Ising models,
and more recently in the context of stochastic homogenization, cf.~e.g.~\cite{Naddaf-Spencer-98,Glotto-11,Glotto-12,GNO-I,Glotto-Neukamm-14,Marahrens-Otto-15,DGO1}.

\medskip

In the context of partial differential equations (PDEs) with random coefficients, we consider random coefficient fields that are defined on $\R^d$ rather than on $\Z^d$.
In this continuum setting, let $A:\R^d\times\Omega\to\R$ be a jointly measurable random field on $\R^d$ (we use a capital letter to emphasize the difference with the discrete case), constructed on a probability space $(\Omega,\A,\p)$.
The Poincaré inequality~\eqref{0-int-0} is then naturally replaced by 
\begin{equation}\label{0-int-3}
\var{Z(A)} \,\le \,C\, \E\bigg[{\int_{\R^d} |\parfct{A,B(x)} Z(A)|^2dx}\bigg],
\end{equation}
where $B(x)$ denotes the unit ball centered at $x\in \R^d$ and where the ``functional derivative'' $\parfct{A,B(x)} Z(A)$ now stands for $\int_{B(x)}|\frac{\partial Z(A)}{\partial A}|$ with $\frac{\partial Z(A)}{\partial A}$ denoting the G\^ateaux (Malliavin type) derivative.
Likewise, the Poincaré inequality~\eqref{0-int-1} is replaced by 
\begin{equation}\label{0-int-4}
\var{Z(A)} \,\le \,C\,\E\bigg[{\int_{\R^d}|\osc_{A,B(x)} Z(A)|^2dx}\bigg],
\end{equation}
where $\osc_{A,B(x)} Z(A)$ now denotes the oscillation of $Z(A)$ wrt variations of $A$ on $B(x)$, that is, formally,
$$\osc_{A,B(x)} Z(A)\,=\,\sup_{A'\,:\, A'|_{\R^d\setminus B(x)}=A|_{\R^d\setminus B(x)}} Z(A')-\inf_{A'\,:\, A'|_{\R^d\setminus B(x)}=A|_{\R^d\setminus B(x)}} Z(A').$$
Such standard functional inequalities~\eqref{0-int-3}--\eqref{0-int-4} are however very restrictive: the random field essentially either has to be Gaussian with integrable covariance (in which case~\eqref{0-int-3} holds) or has to display a product structure (e.g.\@ Poisson point process, in which case~\eqref{0-int-4} holds). 
This rules out most models of interest for heterogeneous materials considered in the applied sciences~\cite{Torquato-02}
and is the starting point for the present series of articles on functional inequalities, which aims at closing this gap.

\medskip

To this aim, we introduce multiscale functional inequalities (MFIs), which are multiscale weighted generalizations of standard functional inequalities (Poincaré, covariance, and logarithmic Sobolev inequalities).
More precisely, given an integrable weight~$\pi:\R_+\to \R_+$, the multiscale versions of~\eqref{0-int-3} and~\eqref{0-int-4} take the form 
\begin{eqnarray} 
\var{Z(A)} &\le & \expec{\int_0^\infty \int_{\R^d} |\parfct{A,B_\ell(x)} Z(A)|^2dx\, (\ell+1)^{-d}\pi(\ell)\,d\ell},\label{0-int-5}
\\
\var{Z(A)} &\le & \expec{ \int_0^\infty \int_{\R^d}|\osc_{A,B_\ell(x)} Z(A)|^2dx\, (\ell+1)^{-d}\pi(\ell)\,d\ell},\label{0-int-6}
\end{eqnarray}
where $B_\ell(x)$ is the ball of radius $\ell$ centered at $x\in \R^d$.
In a nutshell, MFIs are to standard functional inequalities what $\alpha$-mixing conditions are to ensembles with finite range of dependence:
MFIs take into account variations of $A$ on arbitrarily large sets ($\ell\gg 1$) but 
with a decaying weight, which gives them the flexibility to include strongly correlated random fields.
{\color{black}We refer to the companion article~\cite{DG1} for a thorough discussion of the link between the decay of the weight and mixing properties.
Note that a power $(\ell+1)^{-d}$ is singled out from the weight $\pi$ in the notation~\eqref{0-int-5}--\eqref{0-int-6} in order to compensate for the typical size of variations on balls of radius $\ell$: with this choice, the ergodicity of the random field $A$ is precisely guaranteed by the integrability of~$\pi$ (cf.~\cite[Proposition~1.4]{DG1}).}
The aim of the present contribution is to show that important examples of correlated random fields do indeed
satisfy MFIs~\eqref{0-int-5}--\eqref{0-int-6} whereas they do not satisfy their standard versions~\eqref{0-int-3}--\eqref{0-int-4}.
Our approach covers all the models considered in the reference textbook~\cite{Torquato-02} on heterogeneous materials modelling.

\subsection{Constructive approach to multiscale functional inequalities}

Random coefficient fields $A$ considered in~\cite{Torquato-02} for heterogeneous materials modelling have the property of being the form $A=\Phi(A_0)$,
where $A_0$ is a simpler random field with a product structure and where $\Phi$ is a (potentially
complicated) nonlinear
nonlocal transformation.
In particular, standard functional inequalities~\eqref{0-int-3}--\eqref{0-int-4} hold for $\sigma(A_0)$-measurable random variables $Z_0(A_0)$.
The main question we answer in this contribution is under what assumptions on the transformation $\Phi$ and for which weight $\pi$
standard functional inequalities~\eqref{0-int-3}--\eqref{0-int-4} for $A_0$ can be deformed into multiscale functional inequalities \eqref{0-int-5}--\eqref{0-int-6} for $A$.
In view of the relation $Z(A)=Z_0(A_0)$ with $Z_0=Z \circ\Phi$, this amounts to devising an (approximate) chain rule in terms of properties of the transformation $\Phi$, which can quickly become a subtle problem (cf.~e.g.\@ the case of the random parking process below). The weight $\pi$ arises in link with the lack of locality of $\Phi$.

\medskip

Let us give three examples of random coefficient fields $A$ 
that do not satisfy standard functional inequalities but for which we establish multiscale functional inequalities:
\begin{itemize}
\item \emph{Gaussian fields.} Let $A(x):=b(A_1(x))$, where $b$ is a bounded Lipschitz function and $A_1$ is a stationary Gaussian field with covariance function
$\mathcal C:\R^d \to \R$ such that $|\mathcal C(x)|\le c(|x|)$ with $c:\R_+\to \R_+$ differentiable and decreasing. Then the field $A$ satisfies~\eqref{0-int-5}
with weight $\pi(\ell)=-C c'(\ell)$ for some constant $C$ depending only on~$b$.
This constitutes an alternative to Poincar\'e inequalities in terms of Malliavin calculus, cf.~e.g.~\cite{Houdre-PerezAbreu-95,Nourdin-Peccati-12}.
\smallskip\item \emph{Voronoi tessellation of a Poisson point set.} Let $A(x):=\sum_jV_j \mathds1_{C_j}(x)$, where $V_j$ are i.i.d.\@ random variables and where $C_j$ are the cells
of the Voronoi tessellation of $\R^d$ associated with the realization of a Poisson point process of given intensity. Then the field $A$ satisfies  \eqref{0-int-6}
with $\pi(\ell)=C \exp(-\frac1C \ell^d)$ for some constant $C$ depending only on the law of $V$ and on the intensity of the underlying Poisson process.
\smallskip\item \emph{Random parking measure.} Let $A(x):=\alpha+(\beta-\alpha)\sum_j \mathds1_{B_j}(x)$, where $\alpha$ and $\beta$ are deterministic values and where $B_j$ are unit balls centered at the points of a random parking process (formally defined as the thermodynamic limit of a packing process of unit balls at saturation~\cite{Penrose-01}).
Then the field $A$ satisfies  \eqref{0-int-6}
with $\pi(\ell)= C\exp(-\frac1C \ell)$ for some universal constant $C$.
\end{itemize}
As shown in the companion article~\cite{DG1}, the validity of such functional inequalities entails in particular that these three examples of random fields enjoy strong concentration properties (with tail behavior ranging from stretched exponential to Gaussian), although they do not satisfy any standard functional inequality.
Let us briefly indicate how these examples can be viewed as transformations of simpler structures.
First, in the Gaussian example, we write $A=b(\Phi(A_0))$ where~$A_0$ is a Gaussian white noise and where $\Phi$ is some nonlocal linear transformation given as the convolution with a suitable kernel determined by the target covariance~$\Cc$ --- in this case the chain rule is  elementary.
Second, in the example of the random tessellation, we write $A=\Phi(A_0)$ where $A_0$ has a product structure (Poisson point process decorated with the i.i.d.\@ random variables $V_j$'s) and where $\Phi$ is a suitable nonlocal map. Note that the nonlocality of $\Phi$ here depends itself on the realization of the Poisson point process: Voronoi cells are indeed not uniformly bounded and the weight $\pi$ in the multiscale functional inequality~\eqref{0-int-6} is precisely related to the decay of the probability of Voronoi cells with large diameter.
Third, in the example of the random parking measure, we write $A=\Phi(A_0)$ where $A_0$ is a Poisson point process on the extended space $\R^d\times\R_+$ and where $\Phi$ is the nonlinear nonlocal map given by Penrose's graphical construction~\cite{Penrose-01}. The nonlocality of $\Phi$ then depends on the realization of the Poisson point process in a particularly intricate way: the weight in the multiscale functional inequality~\eqref{0-int-6} is related to the so-called stabilization radius introduced by Penrose and Yukich~\cite{Penrose-Yukich-02}.
For the last two examples, we introduce a new general geometric notion of action radius (inspired by~\cite{Penrose-Yukich-02}),
which suitably controls the nonlocality of the transformation $\Phi$ and is the key to establishing (random) approximate chain rules that lead to multiscale functional inequalities~\eqref{0-int-5}--\eqref{0-int-6}.

\medskip

The rest of this article is organized as follows. 
In the rest of this introduction we make precise how multiscale functional inequalities can be used in stochastic homogenization.
In Section~\ref{sec:constructive}, we establish various constructive criteria for multiscale functional inequalities, based on
approximate chain rules in standard functional inequalities.
In Section~\ref{sec:examples}, we apply this constructive approach to all the examples of coefficient fields of the reference textbook~\cite{Torquato-02}, thus addressing
in particular the three examples presented above.

\subsection{Application to stochastic homogenization}\label{sec:homog}

For all $f\in C^\infty_c(\R^d)^d$, we consider the following linear elliptic equation in divergence form,
\begin{equation}\label{0-int-2}
-\nabla \cdot A\nabla u \,=\,\nabla\cdot f\qquad \text{ in $\R^d$},
\end{equation}
with random heterogeneous coefficients $A$.
Stochastic homogenization allows to replace this equation on large scales by an effective equation with deterministic constant coefficients, which constitutes a powerful tool to study composite materials in applied physics and mechanics. 
Developing a quantitative theory of stochastic homogenization (that provides error estimates and characterizes fluctuations) is of utmost importance in those fields.
We are interested in the following three main types of results:
\begin{itemize}
\item[(I)] large-scale regularity properties for the (random) solution $\nabla u$;
\item[(II)] quantitative estimates for the homogenization error;
\item[(III)] characterization of the large-scale fluctuations of $\nabla u$.
\end{itemize}

\medskip

There are two classical settings in which quantitative homogenization results are established: either standard functional inequalities in the probability space (or their multiscale versions introduced here) or standard mixing conditions (e.g.\@ $\alpha$-mixing).
Arguments are typically very different in these two settings.
On the one hand, functional inequalities imply a powerful calculus in the probability space, which is particularly convenient to unravel probabilistic cancellations and substantially simplifies the proofs. Optimal scalings can then easily be captured, but stochastic integrability often remains suboptimal (except for~(I)). We refer to the series of works~\cite{GNO-reg,GNO-quant,Fischer-Otto-16,DGO1,DGO2} by Fischer, Neukamm, Otto, and the authors.
On the other hand, standard mixing conditions require a more involved analysis as they only allow to unravel local cancellations after iteration (cf.~e.g.~the renormalization procedure in~\cite{AKM1} and the notion of approximate locality in~\cite{GO4}). Importantly, such iterations lead to (nearly) optimal stochastic integrability --- in contrast with functional inequalities, which cannot be iterated nicely. A full characterization of fluctuations is however still missing in this setting. We refer to the series of works~\cite{Armstrong-Smart-16,Armstrong-Mourrat-16,AKM1,AKM2,AKM-book} by Armstrong, Kuusi, Mourrat, and Smart, and to~\cite{GO4} by Otto and the second author.
Since some random coefficient fields satisfy only one of those two sets of assumptions, it is important to consider both separately.

\medskip

As shown in this contribution, all the examples of random fields appearing in the reference textbook~\cite{Torquato-02} for heterogeneous materials modelling satisfy multiscale functional inequalities. Since some of them also satisfy $\alpha$-mixing conditions (cf.~\cite[Proposition~1.4]{DG1}), we can compare the outcome of the two corresponding approaches: as explained in~\cite[Section~1.3]{DG1} (see also~\cite[Corollary~8]{GNO-reg}), functional inequalities typically capture finer concentration properties, hence finer stochastic integrability.

\tableofcontents

\addtocontents{toc}{\protect\setcounter{tocdepth}{1}}
\subsection*{Notation}
\begin{itemize}
\item $d$ is the dimension of the ambient space $\R^d$;
\item $C$ denotes various positive constants that only depend on the dimension $d$ and possibly on other controlled quantities; we write $\lesssim$ and $\gtrsim$ for $\le$ and $\ge$ up to such multiplicative constants $C$; we use the notation $\simeq$ if both relations $\lesssim$ and $\gtrsim$ hold; we add a subscript in order to indicate the dependence of the multiplicative constants on other parameters;
\item $Q^k:=[-\frac12,\frac12)^k$ denotes the unit cube centered at $0$ in dimension $k$,
and for all  $x\in\R^k$ and $r>0$ we set $Q^k(x):=x+Q^k$, $Q_r^k:=rQ^k$ and $Q_r^k(x):=x+rQ^k$; when $k=d$ or when there is no confusion possible on the meant dimension, we drop the superscript $k$;
\item we use similar notation for balls, replacing $Q^k$ by $B^k$ (the unit ball in dimension~$k$);
\item the Euclidean distance between  subsets of $\R^d$ is denoted by $d(\cdot,\cdot)$;
\item $\B(\R^k)$ denotes the Borel $\sigma$-algebra on $\R^k$;
\item $\expec{\cdot}$ denotes the expectation, $\var{\cdot}$ the variance, and $\cov{\cdot}{\cdot}$ the covariance in the underlying probability space $(\Omega,\A,\p)$,
and the notation $\expec{\cdot \| \cdot}$ stands for the conditional expectation;
\item for a subset $D$ of a reference set $E$, we let $D^c:=E\setminus D$ denote its complement;
\item for all $a,b\in \R$, we set $a\wedge b:=\min\{a,b\}$, $a\vee b:=\max\{a,b\}$, and $a_+:=a\vee0$;
\item for all matrices $F$, we denote by $F^t$ its transpose matrix;
\item $\lceil a \rceil$ denotes the smallest integer larger or equal to $a$.
\end{itemize}


\addtocontents{toc}{\protect\setcounter{tocdepth}{2}}
\section{Constructive approach to multiscale functional inequalities}\label{sec:constructive}

In this section we consider random fields that can be constructed as transformations of product structures. Under suitable assumptions we describe how the standard Poincaré, covariance, and logarithmic Sobolev inequalities satisfied by the ``hidden'' product structures are deformed into multiscale functional inequalities for the random fields of interest. Various general criteria are established, while the analysis of the examples mentioned in the introduction is postponed to Section~\ref{sec:examples}.

\subsection{Multiscale functional inequalities}\label{chap:spectralgaps}

We start with a precise definition of multiscale functional inequalities.
Let $A:\R^d\times\Omega\to\R$ be a jointly measurable random field on $\R^d$, constructed on some probability space $(\Omega,\A,\p)$. A Poincaré inequality in probability for~$A$ is a functional inequality that allows to control the variance of any $\sigma(A)$-measurable random variable $Z(A)$ in terms of its local dependence on $A$, that is, in terms of some ``derivative'' of $Z(A)$ wrt local restrictions of $A$. 
In the present continuum setting, we consider three possible notions of derivatives.
\begin{itemize}
\item The {\it oscillation} $\osc$ is formally defined by
\begin{multline}
\hspace{1cm}\oscd{A,S}Z(A)~:=~\supessd{A,S}Z(A)-\infessd{A,S}Z(A)\\
\text{``$=$''}~ \supess\left\{Z(A'):A'\in\Mes(\R^d;\R),\,A'|_{\R^d\setminus S}=A|_{\R^d\setminus S}\right\} \\
-\infess\left\{Z(A'):A'\in\Mes(\R^d;\R),\,A'|_{\R^d\setminus S}=A|_{\R^d\setminus S}\right\},\label{e.def-Glauber-formal}
\end{multline}
where the essential supremum and infimum are taken wrt the measure induced by the field $A$ on the space $\Mes(\R^d;\R)$ (endowed with the cylindrical $\sigma$-algebra). This definition \eqref{e.def-Glauber-formal} of $\osc_{A,S}Z(A)$ is not measurable in general, and we rather define
\[\oscd{A,S}Z(A):=\M[Z\|A|_{\R^d\setminus S}]+\M[-Z\|A|_{\R^d\setminus S}]\]
in terms of the conditional essential supremum $\M[\cdot\|A_{\R^d\setminus S}]$ given $\sigma(A|_{\R^d\setminus S})$, as introduced in~\cite{Barron}.
Alternatively, we may simply define $\osc_{A,S}Z(A)$ as the measurable envelope of~\eqref{e.def-Glauber-formal}. These measurable choices are equivalent for the application to stochastic homogenization, and one should  not worry about these measurability issues.
\smallskip\item The (integrated) {\it functional} (or Malliavin type) {\it derivative} $\parfct{}$ is the closest generalization of the usual partial derivatives commonly used in the discrete setting.
Choose an open set $M\subset \Ld^\infty(\R^d)$ containing the realizations of the random field~$A$.
Given a $\sigma(A)$-measurable random variable $Z(A)$ and given an extension $\tilde Z:M\to\R$ of~$Z$, its G\^ateaux derivative $\frac{\partial\tilde Z(A)}{\partial A}\in\Ld^1_\loc(\R^d)$ is defined as follows, for all compactly supported perturbations $B\in\Ld^\infty(\R^d)$,
\[\hspace{1.2cm}\lim_{t\to0}\frac{\tilde Z(A+tB)-\tilde Z(A)}t=\int_{\R^d}B(x)\frac{\partial\tilde Z(A)}{\partial A}(x)\,dx,\]
if the limit exists.
{\color{black}(The extension $\tilde Z$ is only needed to make sure that quantities like $\tilde Z(A+tB)$ make sense for small $t$, while $Z$ is a priori only defined on realizations of~$A$. In the sequel we will always assume that such an extension is implicitly given; this is typically the case in applications in stochastic homogenization.)}
Since we are interested in the local averages of this derivative, we rather define for all bounded Borel subsets $S\subset\R^d$,
\begin{align*}
\hspace{1.2cm}\parfct{A,S}Z(A)=\int_S\Big|\frac{\partial\tilde Z(A)}{\partial A}(x)\Big|dx,
\end{align*}
which is alternatively characterized by
\begin{align*}
\hspace{1.2cm}\parfct{A,S} Z(A)
\,=\,\sup\Big\{ \limsup_{t\downarrow 0} \frac{\tilde Z(A+t B)-\tilde Z(A)}{t} \,:\, \supp{B}\subset S,\,\sup |B|\le 1\Big\}.
\end{align*}
This derivative is additive wrt the set $S$: for all disjoint Borel subsets $S_1,S_2\subset\R^d$ we have $\partial_{A,S_1\cup S_2}^{\operatorname{fct}}Z(A)=\partial_{A,S_1}^{\operatorname{fct}}Z(A)+\partial_{A,S_2}^{\operatorname{fct}}Z(A)$.
{\color{black}\smallskip\item The supremum of the functional derivative is defined as
\begin{align*}
\hspace{1.2cm}\partial_{A,S}^{\operatorname{\sup}}Z(A)~:=~\supessd{A,S}\int_S\Big|\frac{\partial\tilde Z(A)}{\partial A}\Big|.
\end{align*}
Note that there holds $\osc,\parfct{}\lesssim \parsup{}$ provided that $A$ is uniformly bounded.
From the proofs in the companion article~\cite{DG1}, it is clear that multiscale functional inequalities with $\parsup{}$ imply the same concentration properties as the corresponding functional inequalities with $\osc$.}
\end{itemize}

\medskip
Henceforth we use $\tilde\partial$ to denote either $\osc$, $\parfct{}$, or $\parsup{}$. 
We are now in position to define multiscale functional inequalities, which are multiscale weighted versions of standard functional inequalities in the probability space.
\begin{defin}\label{defi:TFI}
Given an integrable function $\pi:\R_+\to\R_+$, we say that $A$ satisfies the {\it multiscale Poincaré inequality} (or spectral gap) ($\tilde \partial$-MSG) {\it with weight $\pi$} if for all $\sigma(A)$-measurable random variables $Z(A)$ we have
\begin{align*}
\var{Z(A)}\,\le \,\expec{\int_0^\infty \int_{\R^d}\Big(\tilde\partial_{A,B_{\ell+1}(x)}Z(A)\Big)^{2}dx\,(\ell+1)^{-d}\pi(\ell)\,d\ell};
\end{align*}
it satisfies the  {\it multiscale covariance inequality \emph{($\tilde\partial$-MCI)} with weight $\pi$} 
if for all $\sigma(A)$-measurable random variables $Y(A),Z(A)$ we have
\begin{multline*}
\cov{Y(A)}{Z(A)}\\
\,\le\, \int_0^\infty \int_{\R^d}\expec{\Big(\tilde\partial_{A,B_{\ell+1}(x)}Y(A)\Big)^2}^{\frac12}\expec{\Big(\tilde\partial_{A,B_{\ell+1}(x)}Z(A)\Big)^2}^{\frac12}dx \,(\ell+1)^{-d}\pi(\ell)\,d\ell;
\end{multline*}
it satisfies the {\it multiscale logarithmic Sobolev inequality \emph{($\tilde\partial$-MLSI)} with weight $\pi$} if for all $\sigma(A)$-measurable random variables $Z(A)$ we have
\begin{align*}
\ent{Z(A)^2}&~:=~\expec{Z(A)^2\log\frac{Z(A)^2}{\expec{Z(A)^2}}}\\
&~\le~ \expec{\int_0^\infty \int_{\R^d}\Big(\tilde\partial_{A,B_{\ell+1}(x)}Z(A)\Big)^{2}dx\,(\ell+1)^{-d}\pi(\ell)\,d\ell}.\qedhere
\end{align*}
\end{defin}

\begin{rem}\label{rem:proofplus}
In each of the examples considered in the sequel, if the functional inequalities ($\tilde\partial$-MSG), ($\tilde\partial$-MCI), or ($\tilde\partial$-MLSI) are proved to hold with some weight $\pi$, then
for all $L\ge1$ the rescaled field $A_L:=A(L\cdot)$ satisfies the same functional inequality with the same weight~$\pi$.
See Remarks~\ref{rem:rescale-1} and~\ref{rem:rescale-2} for detail. This property is used in \cite{GNO-reg}.
\end{rem}

Corresponding standard functional inequalities (standard Poincaré (SG), covariance (CI), and logarith\-mic Sobo\-lev inequality (LSI)) are recovered by taking a compactly supported weight $\pi$, or equivalently, by skipping the integrals over $\ell$ and setting $\ell:=R$ a fixed radius.
Classical arguments yield the following sufficient criterion for standard functional inequalities. A standard proof is included for completeness in Appendix~\ref{sec:appendA} and will be referred to at several places in this contribution.
\begin{prop}\label{prop:criMSG}
Let $A_0$ be a random field on $\R^d$ with values in some measurable space such that restrictions $A_0|_{S}$ and $A_0|_T$ are independent for all disjoint Borel subsets $S,T\subset\R^d$. Let $A$ be a random field on $\R^d$ that is an $R$-local transformation of $A_0$, in the sense that for all $S\subset\R^d$ the restriction $A|_{S}$ is $\sigma(A_0|_{S+B_{R}})$-measurable. Then the field $A$ satisfies  \emph{($\osc$-CI)} and \emph{($\osc$-LSI)} with radius $R+\e$ for all $\e>0$.
\end{prop}
Note that any field satisfying the assumption in the above criterion has finite range of dependence.
Conversely, any field that satisfies~($\osc$-CI) has necessarily finite range of dependence (cf.~\cite[Proposition~1.4(iv)]{DG1}).
One does not expect, however, finite range of dependence to be a sufficient condition for the validity of (SG) in general (compare indeed with the constructions in~\cite{Burton-Goulet-Meester-93,Bradley-94}).

\subsection{Transformation of product structures}\label{eq:transfiid}

Let the random field $A$ on $\R^d$ be $\sigma(\X)$-measurable for some random field $\X$ defined on some measure space $X$ and with values in some measurable space $M$. Assume that we have a partition $X=\biguplus_{x\in\Z^d,t\in\Z^l}X_{x,t}$, on which $\X$ is {\it completely independent}, that is, the restrictions $(\X|_{X_{x,t}})_{x\in\Z^d,t\in\Z^l}$ are all independent.

\medskip

In the sequel, the case $l=0$ (that is, the case when there is no variable $t$) is referred to as the non-projective case, while $l\ge1$ is the projective case. Note that the non-projective case is a particular case of the projective one, simply defining $X_{x,0}=X_x$ and $X_{x,t}=\varnothing$ for all $t\ne0$. The random field $\X$ can be e.g.\@ a random field on $\R^d\times\R^l$ with values in some measure space (choosing $X=\R^d\times\R^l$, $X_{x,t}=Q^d(x)\times Q^l(t)$, and $M$ the space of values), or a random point process (or more generally a random measure) on $\R^d\times\R^l\times X'$ for some measure space $X'$ (choosing $X=\Z^d\times\Z^l\times X'$, $X_{x,t}=\{x\}\times\{t\}\times X'$, and $M$ the space of measures on $Q^d\times Q^l\times X'$).

\medskip

Let $\X'$ be some given i.i.d.\@ copy of $\X$. For all $x,t$, we define a perturbed random field $\X^{x,t}$ by setting $\X^{x,t}|_{X\setminus X_{x,t}}=\X|_{X\setminus X_{x,t}}$ and $\X^{x,t}|_{X_{x,t}}=\X'|_{X_{x,t}}$. By complete independence, the random fields $\X$ and $\X^{x,t}$ (resp.\@ $A=A(\X)$ and $A(\X^{x,t})$)  have the same law. Arguing as in the proof of Proposition~\ref{prop:criMSG} (cf.~\eqref{eq:tens-cov} and~\eqref{eq:tens-ent}
in Appendix~\ref{sec:appendA}), the complete independence assumption ensures that $\X$ satisfies the following standard functional inequalities, which are variations around the Efron-Stein inequality~\cite{Efron-Stein-81,Steele-86}.
\begin{prop}\label{prop:sgindep}
For all $\sigma(\X)$-measurable random variables $Y(\X),Z(\X)$, we have
\begin{gather}
\var{Y(\X)}\le\frac12\,\expecM{\sum_{x\in\Z^d}\sum_{t\in\Z^l}\big(Y(\X)-Y(\X^{x,t})\big)^2},\label{eq:var-X}\\
\ent{Y(\X)}\le2\,\expecM{\sum_{x\in\Z^d}\sum_{t\in\Z^l}\supessd{\X'}\big(Y(\X)-Y(\X^{x,t})\big)^2},\label{eq:ent-X}\\
\cov{Y(\X)}{Z(\X)}\le \frac12\sum_{x\in\Z^d}\sum_{t\in\Z^l}\expec{ \big(Y(\X)-Y(\X^{x,t})\big)^2}^\frac12\expec{ \big(Z(\X)-Z(\X^{x,t})\big)^2}^\frac12.\label{eq:cov-X}
\end{gather}\qedhere
\end{prop}

We briefly comment on the form of the logarithmic Sobolev inequality. 
A common difficulty when applying \eqref{eq:ent-X} stems from the supremum in the RHS (compared to the
variance estimate \eqref{eq:var-X}). 
In~\cite{BLM-03}, Boucheron, Lugosi and Massart introduced a variant of    \eqref{eq:ent-X} 
in exponential form that avoids taking a supremum (see also~\cite{Wu00} for the Poisson process, and its subtle applications~\cite{BachPec} to stochastic geometry).
It seems that the approach we develop below based on the notion of action radius and conditioning 
behaves badly in exponential form, and we are currently unable to combine it with the techniques of~\cite{BLM-03}.

\subsection{Abstract criteria and action radius}\label{chap:condunif-pasunif}

We now describe general situations for which the functional inequalities for the ``hidden'' product structure $\X$ are deformed into multiscale inequalities for the random field $A$.
We distinguish between the following two cases:
\begin{itemize}
\item {\it Deterministic localization}, that is, when the random field $A$ is a deterministic convolution of some product structure, so that the ``dependence pattern'' is prescribed deterministically a priori. It leads to multiscale functional inequalities with the functional derivative $\parfct{}$.
\smallskip\item {\it Random localization}, that is, when the ``dependence pattern'' is encoded by the underlying product structure $\X$ itself (and therefore may depend on the realization, whence the terminology ``random''). The localization of the dependence pattern is then measured in terms of what we call the {\it action radius} and it leads to multiscale inequalities with the derivative $\osc$.  This generalizes the idea of local transformations in Proposition~\ref{prop:criMSG}, which would indeed correspond to the case of a deterministic bound on the action radius.
\end{itemize}

The case of {\it deterministic} localization mainly concerns Gaussian fields, which have been thoroughly studied in the literature. Multiscale functional inequalities for such random fields constitute a possible alternative to functional inequalities in terms of Malliavin calculus, cf.~e.g.~\cite{Houdre-PerezAbreu-95,Nourdin-Peccati-12}.
As emphasized in the companion article~\cite[Appendix~A]{DG3}, multiscale functional inequalities can then indeed be directly deduced from the corresponding Malliavin results: the key relies on a deterministic radial change of variables to reformulate Hilbert norms encoding the covariance structure into multiscale weighted norms.
To remain in the spirit of our general approach, a self-contained proof is included in Appendix~\ref{chap:condunif} where multiscale functional inequalities are established via the deformation of standard functional inequalities for i.i.d.\@ Gaussian sequences.

\medskip

In the rest of this section, we focus on the more original setting of {\it random} localization, which involves a \emph{random} change of variable due to the randomness of the dependence pattern.
We use the notation of Section~\ref{eq:transfiid}: $A$ is a $\sigma(\X)$-measurable random field on $\R^d$, where $\X$ is a completely independent random field on some measure space $X=\biguplus_{x\in\Z^d,t\in\Z^l}X_{x,t}$ with values in some measurable space $M$.
The following definition is inspired by the notion of stabilization radius first introduced by Lee~\cite{Lee-97,Lee-99} and crucially used in the works by Penrose, Schreiber, and Yukich on random sequential adsorption processes~\cite{Penrose-Yukich-02,Penrose-05,PY-05,Schreiber-Penrose-Yukich-07}.
\begin{defin}\label{def:ar}
Given an i.i.d.\@ copy $\X'$ of the field $\X$, an {\it action radius for $A$ wrt $\X$ on $X_{x,t}$} (with reference perturbation $\X'$), if it exists, is defined as a nonnegative $\sigma(\X,\X')$-measurable random variable $\rho$ such that we have a.s.,
\[\left.A(\X^{x,t})\right|_{\R^d\setminus (Q(x)+B_\rho)}=\left.A(\X)\right|_{\R^d\setminus(Q(x)+B_\rho)},\]
where the perturbed random field $\X^{x,t}$ is defined as before by $\X^{x,t}|_{X\setminus X_{x,t}}:=\X|_{X\setminus X_{x,t}}$ and $\X^{x,t}|_{X_{x,t}}:=\X'|_{X_{x,t}}$.
\end{defin}
Note that if $\X=A_0$ is a random field on $\R^d$, and if for some $R>0$ the random field $A$ is an $R$-local transformation of $A_0$ in the sense of Proposition~\ref{prop:criMSG}, then the constant $\rho=R$ is an action radius for $A$ wrt $A_0$ on any set.
Reinterpreted in the case when $\X=\Pc$ is a random point process on $\R^d\times\R^{l}\times X'$ for some measure space $X'$, the above definition takes on the following guise: given a subset $E\times F\subset \R^d\times \R^{l}$ and given an i.i.d.\@ copy~$\Pc'$ of~$\Pc$, an {\it action radius for $A$ wrt $\Pc$ on $E\times F$}, if it exists, is a nonnegative random variable $\rho$ such that we have a.s.,
\[\left.A\Big(\big(\Pc\setminus (E\times F\times X')\big)\bigcup \big(\Pc'\cap (E\times F\times X')\big)\Big)\right|_{\R^d\setminus (E+B_\rho)}=\left.A(\Pc)\right|_{\R^d\setminus (E+B_\rho)}.\]

We display two general criteria, Theorems~\ref{th:ar} and~\ref{th:ar-rpm} below,
for the validity of multiscale functional inequalities in terms of the properties of an action radius. 
The argument consists in conditioning wrt the action radius and then using some independence in order to avoid the use of H\"older's inequality (which would lead to a loss of integrability in the functional inequalities). We start in Theorem~\ref{th:ar} with the simplest dependence pattern (cf.\@ independence assumption (c) below for the action radius), which already encompasses some examples of interest (like spherical inclusions centered at the points of a Poisson point process with i.i.d.\@ random radii, cf.\@ Section~\ref{chap:genincl}).
Note that the additional condition for the validity of the multiscale logarithmic Sobolev inequality below is rather stringent.
\begin{theor}\label{th:ar}
Let the fields $A,\X$ be as above. Given an i.i.d.\@ copy $\X'$ of the field $\X$, assume that
\begin{enumerate}[(a)]
\item For all $x,t$, there exists an action radius $\rho_{x,t}$ for $A$ wrt $\X$ in $X_{x,t}$.
\smallskip\item The transformation $A$ of $\X$ is stationary, that is, the random fields $A(\X(\cdot+x,\cdot))$ and $A(\X)(\cdot+x)$ have the same law for all $x\in\Z^d$.
%
\item For all $x,t$ the action radius $\rho_{x,t}$ is independent of $A|_{\R^d\setminus (Q(x)+B_{f(\rho_{x,t})})}$ for some \emph{influence function} $f:\R_+\to\R_+$ with $f(u)\ge u$ for all $u$.
\end{enumerate}
{\color{black}With the convention $\frac00=0$, set
\begin{gather*}
\pi(\ell):=(\ell+1)^{d}\sum_{t\in\Z^l}\tilde\pi(t,\ell),\qquad
\pi(t,\ell):=\pr{\X^{0,t}\ne \X}\frac{\prC{\ell-1\le\rho_{0,t}<\ell}{\X^{0,t}\ne \X}}{\pr{\rho_{0,t}<\ell}}.
\end{gather*}
Then for all $\sigma(A)$-measurable random variables $Z(A)$ we have
%
\begin{equation}\label{eq:arwsg2}
\var{Z(A)}\le\frac12\,\expecM{\int_0^\infty\int_{\R^d}\,\Big(\oscd{A,B_{\sqrt d(f(\ell)+1)}(y)}Z(A)\Big)^2dx\,(\ell+1)^{-d} \pi(\ell)\,d\ell},
\end{equation}
If in addition the random variable $\rho_{x,t}$ is $\sigma(\X)$-measurable for all $x,t$, there holds
\begin{equation}
\ent{Z(A)}\le2\,\expec{ \int_{0}^\infty \int_{\R^d}\Big(\oscd{A,B_{\sqrt{d}(f(\ell)+1)}(x)}Z(A)\Big)^2dx\,(\ell+1)^{-d}\pi(\ell)\,d\ell}.\label{eq:arlsi2}\qedhere
\end{equation}}
\end{theor}

\begin{rem}\label{rem:ar-ci}
{\color{black}
Rather starting from the covariance form~\eqref{eq:cov-X}, the proof below further yields, next to~\eqref{eq:arwsg2}, for all $\sigma(A)$-measurable random variables $Y(A),Z(A)$,
\begin{multline}\label{eq:arci2}
\qquad\cov{Y(A)}{Z(A)}\le\frac12\sum_{t\in\Z^l}\int_{\R^d}\bigg(\int_0^\infty\pi(t,\ell)\,\expec{\Big(\oscd{A,B_{\sqrt d(f(\ell)+1)}(x)}Y(A)\Big)^2}d\ell\bigg)^\frac12\\
\times\bigg(\int_0^\infty\pi(t,\ell)\,\expec{\Big(\oscd{A,B_{\sqrt d(f(\ell)+1)}(x)}Z(A)\Big)^2}d\ell\bigg)^\frac12dx.
\end{multline}
This can in general not be usefully recast into the canonical form of the multiscale covariance inequality from Definition~\ref{defi:TFI}, except in some examples (cf.~e.g.~Remark~\ref{rem:cov-alm} and Proposition~\ref{prop:genincl-loss}(i) below).}
\end{rem}

{\color{black}In many cases of interest, the above independence assumption~(c) is however too stringent: making $\rho_{x,t}$ independent of $A|_{\R^d\setminus(Q(x)+B_{\rho^*})}$ may indeed require to construct $\rho^*$ as a larger random variable that is not $\sigma(\rho_{x,t})$-measurable.}
We turn to a more complex situation when the dependence pattern is still sufficiently well-controlled in terms of a \emph{family} of successive action radii. The measurability assumption~(c) below mimics the dependence properties of the action radius for the Voronoi tessellation of a Poisson point process (cf.~Section~\ref{chap:tessel}) and for the random parking process (cf.~Section~\ref{chap:RPM}).

\begin{theor}\label{th:ar-rpm}
Let $A=A(\X)$ be a $\sigma(\X)$-measurable random field on $\R^d$, where $\X$ is a completely independent random field on some measure space $X=\biguplus_{x\in\Z^d}X_{x}$ with values in some measurable space $M$. For all $x\in\Z^d$, $\ell\in\N$, set $X_x^\ell:=\bigcup_{y\in\Z^d:|x-y|_\infty\le\ell}X_y$.
Given an i.i.d.\@ copy $\X'$ of the field $\X$, let the perturbed field $\X^{x,\ell}$ be defined by
\[\X^{x,\ell}|_{X\setminus X_x^\ell}=\X|_{X\setminus X_x^\ell},\qquad\text{and}\qquad\X^{x,\ell}|_{X_x^\ell}=\X'|_{X_x^\ell},\]
and assume that
\begin{enumerate}[(a)]
\item For all $x,\ell$, there exists an action radius $\rho_x^\ell$ for $A$ wrt $\X$ in $X_x^\ell$, that is, a nonnegative random variable $\rho_x^\ell$ such that we have a.s.,
\[A(\X^{x,\ell})|_{\R^d\setminus (Q_{2\ell+1}(x)+B_{\rho_x^\ell})}=A(\X)|_{\R^d\setminus (Q_{2\ell+1}(x)+B_{\rho_x^\ell})}.\]
\item The transformation $A$ of $\X$ is stationary, that is, the random fields $A(\X(\cdot+x,\cdot))$ and $A(\X)(\cdot+x)$ have the same law for all $x\in\Z^d$.
\item For all $x,\ell$, the random variable $\rho_x^\ell$ is $\sigma\Big(\X\big|_{X_x^{\ell+\rho_x^\ell}\setminus X_x^\ell}\Big)$-measurable.\footnote{{\color{black}This is understood as follows: for all $r\ge0$ the event $\{\rho_x^\ell>r\}$ belongs to $\sigma\Big(\X\big|_{X_x^{\ell+r}\setminus X_x^\ell}\Big)$.}}\\
(In particular, for all $x,\ell,R$, given the event $\rho_x^\ell\le R$, the random variables $\rho_x^\ell$ and~$\rho_x^{\ell+R}$ are independent.)
\end{enumerate}
Assume that $R\ge1$ can be chosen large enough so that
\begin{align}\label{eq:choice-R}
\sup_{\ell\ge R}\,\prm{\rho_x^\ell\ge\ell}\le\tfrac14,
\end{align}
let $\pi_0:\R_+\to\R_+$ be a non-increasing function such that $\pr{\frac14\ell\le\rho_x^{\ell_0}<\ell}\le\pi_0(\ell)$ holds for all $0\le\ell_0\le\frac14\ell$, and define the weight
{\color{black}\[\pi(\ell):=(\ell+1)^{d}\begin{cases}1,&\text{if $\ell\le 4R$};\\8\ell^{-1}\pi_0(\frac12\ell),&\text{if $\ell>4R$}.\end{cases}\]
Then for all $\sigma(A)$-measurable random variables $Z(A)$ we have
\begin{eqnarray}
\var{Z(A)}&\le&\frac12\,\expec{\int_0^\infty\int_{\R^d}{\Big(\oscd{A,B_{\sqrt d(\ell+1)}(x)}Z(A)\Big)^2}dx\,(\ell+1)^{-d}\pi(\ell) d\ell},\label{eq:AG-SG}\\
\ent{Z(A)}&\le&2\,\expec{\int_0^\infty\int_{\R^d}{\Big(\oscd{A,B_{\sqrt d(\ell+1)}(x)}Z(A)\Big)^2}dx\,(\ell+1)^{-d}\pi(\ell) d\ell}.\label{eq:AG-LSI}
\end{eqnarray}}
\end{theor}

{\color{black}\begin{rem}\label{rem:cov-alm}
Rather starting from the covariance form~\eqref{eq:cov-X}, the proof further yields, for all $\sigma(A)$-measurable random variables $Y(A),Z(A)$,
\begin{multline*}
\cov{Y(A)}{Z(A)}\le \frac12\int_{\R^d}\bigg(\int_0^\infty\expec{\Big(\oscd{A,B_{\sqrt d(\ell+1)}(x)}Y(A)\Big)^2}(\ell+1)^{-d}\pi(\ell) d\ell\bigg)^\frac12\\
\times\bigg(\int_0^\infty\expec{\Big(\oscd{A,B_{\sqrt d(\ell+1)}(x)}Z(A)\Big)^2}(\ell+1)^{-d}\pi(\ell) d\ell\bigg)^\frac12dx.
\end{multline*}
In general this cannot be recast into the canonical form of the multiscale covariance inequality from Definition~\ref{defi:TFI} except if the weight $\pi$ is decaying enough: If $\pi$ is non-increasing and satisfies $\int_0^\infty(\ell+1)^{-\frac d2}\pi(\ell)^\frac12d\ell<\infty$, it indeed follows from the discrete $\ell^1$--$\ell^2$ inequality that
\begin{multline*}
\cov{Y(A)}{Z(A)}\lesssim_\pi \int_0^\infty\int_{\R^d}\expec{\Big(\oscd{A,B_{\sqrt d(\ell+3)}(x)}Y(A)\Big)^2}^\frac12\\
\times \expec{\Big(\oscd{A,B_{\sqrt d(\ell+3)}(x)}Z(A)\Big)^2}^\frac12dx\,(\ell+1)^{-\frac d2}\pi(\ell)^\frac12 d\ell,
\end{multline*}
where the square root on the weight is not harmful when $\pi$ has superalgebraic decay.
\end{rem}}

\begin{rem}\label{rem:rescale-1}
We briefly address the claim contained in Remark~\ref{rem:proofplus} in the context of examples of random fields with random localization.
By definition, for all $L\ge1$, an action radius for $A$ wrt $\X$ on $X_{0,t}$ is still an action radius for the rescaled field $A_L:=A(L\cdot)$ wrt $\X$ on $X_{0,t}$.
This proves that in Theorems~\ref{th:ar} and~\ref{th:ar-rpm} any result stated for the field $A$ also holds in the very same form (with the same constants and weights) for $A_L$ with $L\ge1$.
\end{rem}

We start with the proof of Theorem~\ref{th:ar}.
\begin{proof}[Proof of Theorem~\ref{th:ar}]
Recall that for all $x,t$ the perturbed random field $\X^{x,t}$ is defined by $\X^{x,t}|_{X\setminus X_{x,t}}=\X|_{X\setminus X_{x,t}}$ and $\X^{x,t}|_{X_{x,t}}=\X'|_{X_{x,t}}$. By complete independence of $\X$, the fields $\X$ and $\X^{x,t}$ (hence $A=A(\X)$ and $A(\X^{x,t})$) have the same law. {\color{black}By the stationarity assumption~(b) for~$A$, the action radii can be chosen such that the law of $\rho_x^\ell$ is independent of $x$.}
The strategy of the proof consists in deforming the functional inequalities of Proposition~\ref{prop:sgindep} wrt the transformation $A(\X)$ in terms of the action radii. We split the proof into two steps.

\medskip

{\color{black}
\step1 Proof of the Poincaré inequality~\eqref{eq:arwsg2}.

\noindent
We start from~\eqref{eq:var-X} in form of
\begin{equation}\label{eq:var-X-re0}
\var{Z(A)}\le\frac12\sum_{x\in\Z^d}\sum_{t\in\Z^l}\expec{\big(Z(A)-Z(A(\X^{x,t}))\big)^2},
\end{equation}
and for all $x,t$ we consider the following decomposition, conditioning wrt the values of the action radius $\rho_{x,t}$,
\begin{equation*}
\expec{\big(Z(A)-Z(A(\X^{x,t}))\big)^2}=\int_0^\infty \expec{\big(Z(A)-Z(A(\X^{x,t}))\big)^2\mathds1_{\ell-1\le \rho_{x,t}< \ell}}d\ell,
\end{equation*}
%
Recalling that the influence function $f$ satisfies $f(u)\ge u$ for all $u$, we find
\begin{eqnarray*}
\lefteqn{\expec{\big(Z(A)-Z(A(\X^{x,t}))\big)^2}}\\
&=&\int_0^\infty \expec{\big(Z(A)-Z(A(\X^{x,t}))\big)^2\mathds1_{\X|_{X_{x,t}}\ne \X'|_{X_{x,t}}}\mathds1_{\ell-1\le \rho_{x,t}< \ell}}d\ell\\
&\le&\int_0^\infty \expec{\Big(\oscd{A,Q_{2f(\ell)+1}(x)}Z(A)\Big)^2\mathds1_{\X|_{X_{x,t}}\ne \X'|_{X_{x,t}}}\mathds1_{\ell-1\le \rho_{x,t}< \ell}}d\ell\\
&=&\int_0^\infty \expeC{\Big(\oscd{A,Q_{2f(\ell)+1}(x)}Z(A)\Big)^2\mathds1_{\X|_{X_{x,t}}\ne \X'|_{X_{x,t}}}\mathds1_{\rho_{x,t}\ge\ell-1}}{\rho_{x,t}<\ell}\pr{\rho_{x,t}< \ell}d\ell.
\end{eqnarray*}
By definition, given $\rho_{x,t}<\ell$, the restriction $A|_{\R^d\setminus Q_{2f(\ell)+1}(x)}$ is independent of $\X|_{X_{x,t}}$ and $\X'|_{X_{x,t}}$. In addition, by assumption~(c), given $\rho_{x,t}<\ell$, the restriction $A|_{\R^d\setminus Q_{2f(\ell)+1}(x)}$ is independent of $\rho_{x,t}$. We may thus deduce
\begin{eqnarray*}
\lefteqn{\expec{\big(Z(A)-Z(A(\X^{x,t}))\big)^2}}\\
&\le&\int_0^\infty \expeC{\Big(\oscd{A,Q_{2f(\ell)+1}(x)}Z(A)\Big)^2}{\rho_{x,t}<\ell}\pr{\ell-1\le\rho_{x,t}<\ell,\,\X|_{X_{x,t}}\ne \X'|_{X_{x,t}}}d\ell\\
&\le&\int_0^\infty \expec{\Big(\oscd{A,Q_{2f(\ell)+1}(x)}Z(A)\Big)^2}\frac{\pr{\ell-1\le\rho_{x,t}<\ell,\,\X|_{X_{x,t}}\ne \X'|_{X_{x,t}}}}{\pr{\rho_{x,t}<\ell}}\,d\ell.
\end{eqnarray*}
%
%
By stationarity of the action radii, this turns into
\begin{multline}\label{eq:step2-ar-1}
\expec{\big(Z(A)-Z(A(\X^{x,t}))\big)^2}\le
\int_0^\infty\expec{\Big(\oscd{A,Q_{2f(\ell)+1}(x)}Z(A)\Big)^2}\pr{\X|_{X_{0,t}}\ne \X'|_{X_{0,t}}}\\
\times \frac{\prC{\ell-1\le \rho_{0,t}< \ell}{\X|_{X_{0,t}}\ne \X'|_{X_{0,t}}}}{\pr{\rho_{0,t}< \ell}}\,d\ell.
\end{multline}
%
%
Injecting this into~\eqref{eq:var-X-re0} and using the definition of the weight $\pi$ in the statement, we obtain
\[\var{Z(A)}\le\frac12\int_0^\infty(\ell+1)^{-d}\pi(\ell)\sum_{x\in\Z^d}\expec{\Big(\oscd{A,Q_{2f(\ell)+1}(x)}Z(A)\Big)^2}d\ell.\]
Bounding sums by integrals and replacing cubes by balls, the conclusion~\eqref{eq:arwsg2} follows.

\medskip
\step2 Proof of the logarithmic Sobolev inequality~\eqref{eq:arlsi2}. 

\noindent
We rather start from~\eqref{eq:ent-X} in form of
\begin{equation}\label{eq:var-X-re}
\ent{Z(A)}\le2\sum_{x\in\Z^d}\sum_{t\in\Z^l}\expec{\supess_{\X'}\big(Z(A)-Z(A(\X^{x,t}))\big)^2},
\end{equation}
and for all $x,t$ we write, conditioning wrt the values of the action radius $\rho_{x,t}$,
\begin{multline*}
\expec{\supess_{\X'}\big(Z(A)-Z(A(\X^{x,t}))\big)^2}\\
\le\int_0^\infty\expec{\supessd{\X'}\Big(\big(Z(A(\X))-Z(A(\X^{x,t}))\big)^2\mathds1_{\ell-1\le \rho_{x,t}<\ell}\Big)}d\ell\\
\le\int_0^\infty\expec{\Big(\oscd{A,Q_{2\ell+1}(x)}Z(A)\Big)^2\supessd{\X'}\big(\mathds1_{\ell-1\le \rho_{x,t}<\ell}\big)}d\ell.
\end{multline*}
If the random variable $\rho_{x,t}$ is $\sigma(\X)$-measurable, this simply becomes
\[\expec{\supess_{\X'}\big(Z(A)-Z(A(\X^{x,t}))\big)^2}
\le\int_0^\infty\expec{\Big(\oscd{A,Q_{2\ell+1}(x)}Z(A)\Big)^2\mathds1_{\ell-1\le \rho_{x,t}<\ell}}d\ell,\]
and the conclusion~\eqref{eq:arlsi2} follows as in Step~1.}
%
\end{proof}

Next, we turn to the proof of Theorem~\ref{th:ar-rpm}.

\begin{proof}[Proof of Theorem~\ref{th:ar-rpm}]
We only prove the Poincaré inequality~\eqref{eq:AG-SG}. The proof of the logarithmic Sobolev inequality~\eqref{eq:AG-LSI} is similar, rather starting from~\eqref{eq:ent-X}.
For all $x$, let the field $\X^x$ be defined by $\X^x|_{X\setminus X_{x}}=\X|_{X\setminus X_{x}}$ and $\X^x|_{X_{x}}=\X'|_{X_{x}}$, and recall that the Poincaré inequality~\eqref{eq:var-X} for $\X$ takes the form
\begin{align*}
\var{Z(A)}&\le\frac12\sum_{x\in\Z^d}\expec{\big(Z(A)-Z(A(\X^x))\big)^2}.
\end{align*}
{\color{black}Bounding sums by integrals and replacing cubes by balls, the conclusion~\eqref{eq:AG-SG} then follows provided that we prove for all $x\in\Z^d$,
\begin{align}\label{eq:AG-centered}
\expec{\big(Z(A)-Z(A(\X^x))\big)^2}\le\int_0^\infty \expec{\Big(\oscd{A,Q_{2\ell+1}(x)}Z(A)\Big)^2}(\ell+1)^{-d}\pi(\ell)\,d\ell.
\end{align}}
Without loss of generality, it suffices to consider the case $x=0$. Moreover, by an approximation argument, we may assume that the random variable $Z(A)$ is bounded.
We use the shorthand notation $\rho(r):=r+\rho_0^r$ and $\oscd{r}:=\oscd{A,Q_{2r+1}}$. The choice~\eqref{eq:choice-R} of $R$ then takes the form
\begin{align}\label{eq:choice-R2}
\sup_{\ell\ge R}\,\prm{\rho(\ell)\ge2\ell}\le\tfrac14.
\end{align}
We split the proof into two steps.

\medskip

\step1 Conditioning argument.\\
In this step, we prove for all $r_2\ge 2r_1\ge2R$,
\begin{multline}\label{eq:AG-intermed}
\expec{\Big(\oscd{r_2}Z(A)\Big)^2\mathds1_{\frac12r_2\le\rho(r_1)<r_2}}\le2\,\pr{\tfrac12 r_2\le\rho(r_1)<r_2}\\
\times\bigg(\expec{\Big(\oscd{2r_2}Z(A)\Big)^2}+\sum_{\ell=2}^\infty\expec{\Big(\oscd{2^\ell r_2}Z(A)\Big)^2\mathds1_{2^{\ell-1}r_2\le \rho(r_2)<2^{\ell}r_2}}\bigg).
\end{multline}
Conditioning the LHS wrt the value of $\rho(r_2)$, we decompose
\begin{multline}\label{eq:decomp-AG-2}
\expec{\Big(\oscd{r_2}Z(A)\Big)^2\mathds1_{\frac12r_2\le\rho(r_1)<r_2}}\,\le\,\expec{\Big(\oscd{r_2}Z(A)\Big)^2\mathds1_{\frac12r_2\le\rho(r_1)<r_2}\mathds1_{\rho(r_2)<2r_2}}\\
+\sum_{\ell=2}^\infty\expec{\Big(\oscd{r_2}Z(A)\Big)^2\mathds1_{\frac12r_2\le\rho(r_1)<r_2}\mathds1_{2^{\ell-1}r_2\le\rho(r_2)<2^{\ell}r_2}}.
\end{multline}
We separately estimate the two RHS terms and we start with the first. For that purpose, note that the definition of $\rho$ and assumption~(c) ensure that, given $\rho(r_1)\le r_2$ and $\rho(r_2)\le r_3$, the random variable $\rho(r_1)$ is independent of $\oscd{r_3}Z(A)$. This observation directly yields
\begin{eqnarray*}
\lefteqn{\expec{\Big(\oscd{r_2}Z(A)\Big)^2\mathds1_{\frac12r_2\le\rho(r_1)<r_2}\mathds1_{\rho(r_2)<2r_2}}}\\
&\le&\expeC{\Big(\oscd{2r_2}Z(A)\Big)^2\mathds1_{\rho(r_1)\ge\frac12r_2}}{\rho(r_1)<r_2,\,\rho(r_2)<2r_2}\prm{\rho(r_1)<r_2,\,\rho(r_2)<2r_2}\\
&\le&\expec{\Big(\oscd{2r_2}Z(A)\Big)^2}\frac{\prm{\frac12 r_2\le\rho(r_1)<r_2}}{\pr{\rho(r_1)<r_2,\,\rho(r_2)<2r_2}}\\
&\le&\expec{\Big(\oscd{2r_2}Z(A)\Big)^2}\frac{\prm{\frac12 r_2\le\rho(r_1)<r_2}}{1-\pr{\rho(r_1)\ge r_2}-\pr{\rho(r_2)\ge2r_2}}.
\end{eqnarray*}
For $r_2\ge2r_1\ge2R$, the choice~\eqref{eq:choice-R2} of $R$ yields
\[\pr{\rho(r_1)\ge r_2}+\pr{\rho(r_2)\ge2r_2}\le\pr{\rho(r_1)\ge 2r_1}+\pr{\rho(r_2)\ge2r_2}\le\tfrac12,\]
so that the above takes the simpler form 
\begin{multline}\label{eq:bound-base}
{\expec{\Big(\oscd{r_2}Z(A)\Big)^2\mathds1_{\frac12r_2\le\rho(r_1)<r_2}\mathds1_{\rho(r_2)<2r_2}}}\\
\le2\,\expec{\Big(\oscd{2r_2}Z(A)\Big)^2}\prm{\tfrac12r_2\le\rho(r_1)<r_2}.
\end{multline}
{\color{black}We turn to the second RHS term in~\eqref{eq:decomp-AG-2}.
Recalling that assumption~(c) ensures that given $\rho(r_1)\le r_2$ the random variables $\rho(r_1)$ and $\rho(r_2)$ are independent, we similarly obtain
\begin{eqnarray*}
\lefteqn{\expec{\Big(\oscd{r_2}Z(A)\Big)^2\mathds1_{\frac12r_2\le\rho(r_1)<r_2}\mathds1_{2^{\ell-1}r_2\le\rho(r_2)<2^{\ell}r_2}}}\\
&\le&\expeC{\Big(\oscd{2^\ell r_2}Z(A)\Big)^2\mathds1_{\rho(r_1)\ge\frac12 r_2}\mathds1_{\rho(r_2)\ge2^{\ell-1}r_2}}{\rho(r_1)<r_2,\,\rho(r_2)<2^\ell r_2}\\
&&\qquad\times\prm{\rho(r_1)<r_2,\,\rho(r_2)<2^\ell r_2}\\
&\le&\expec{\Big(\oscd{2^\ell r_2}Z(A)\Big)^2\mathds1_{2^{\ell-1}r_2\le \rho(r_2)<2^\ell r_2}}\frac{\pr{\frac12 r_2\le\rho(r_1)<r_2}}{\pr{\rho(r_1)<r_2,\,\rho(r_2)<2^\ell r_2}}\\
&\le&\expec{\Big(\oscd{2^\ell r_2}Z(A)\Big)^2\mathds1_{2^{\ell-1}r_2\le \rho(r_2)<2^{\ell}r_2}}\frac{\pr{\frac12 r_2\le\rho(r_1)<r_2}}{1-\pr{\rho(r_1)\ge r_2}-\pr{\rho(r_2)\ge 2^\ell r_2}}.
\end{eqnarray*}}
With the choice~\eqref{eq:choice-R2} of $R$, for $r_2\ge 2r_1\ge2R$ and $\ell\ge1$, this turns into
\begin{multline*}
{\expec{\Big(\oscd{r_2}Z(A)\Big)^2\mathds1_{\frac12r_2\le\rho(r_1)<r_2}\mathds1_{2^{\ell-1}r_2\le\rho(r_2)<2^{\ell}r_2}}}\\
\le2\,\expec{\Big(\oscd{2^\ell r_2}Z(A)\Big)^2\mathds1_{2^{\ell-1}r_2\le \rho(r_2)<2^{\ell}r_2}}\pr{\tfrac12 r_2\le\rho(r_1)<r_2}.
\end{multline*}
Combining this with~\eqref{eq:decomp-AG-2} and~\eqref{eq:bound-base}, the conclusion~\eqref{eq:AG-intermed} follows.

\medskip
\step{2} Proof of~\eqref{eq:AG-centered}.\\
Conditioning the LHS of~\eqref{eq:AG-centered} wrt the value of the action radius $\rho(0)$, we obtain
\[\expec{\big(Z(A)-Z(A(\X^x))\big)^2}\le\expec{\Big(\oscd{R}Z(A)\Big)^2}+\sum_{\ell=1}^\infty\expec{\Big(\oscd{2^{\ell}R}Z(A)\Big)^2\mathds1_{2^{\ell-1}R\le\rho(0)<2^\ell R}}.\]
We now iteratively apply~\eqref{eq:AG-intermed} to estimate the last RHS terms: with the short-hand notation $\pi(r_2;r_1):=\prm{\frac12 r_2\le \rho(r_1)<r_2}$, we obtain for all $n\ge1$,
\begin{multline*}
\expec{\big(Z(A)-Z(A(\X^x))\big)^2}\le\expec{\Big(\oscd{R}Z(A)\Big)^2}+2\sum_{\ell_1=1}^\infty\pi(2^{\ell_1} R;0)\,\expec{\Big(\oscd{2^{\ell_1+1}R}Z(A)\Big)^2}\\
+2^2\sum_{\ell_1=1}^\infty\pi(2^{\ell_1} R;0)\sum_{\ell_2=\ell_1+2}^\infty\pi(2^{\ell_2}R;2^{\ell_1}R)\,\expec{\Big(\oscd{2^{\ell_2+1}R}Z(A)\Big)^2}+\ldots\\
+2^n\sum_{\ell_1=1}^\infty\pi(2^{\ell_1} R;0)\sum_{\ell_2=\ell_1+2}^\infty\!\pi(2^{\ell_2}R;2^{\ell_1}R)\ldots\!\!\sum_{\ell_n=\ell_{n-1}+2}^\infty\!\pi(2^{\ell_n}R;2^{\ell_{n-1}}R)\,\expec{\Big(\oscd{2^{\ell_n+1}R}Z(A)\Big)^2}\\
+2^n\sum_{\ell_1=1}^\infty\pi(2^{\ell_1} R;0)\sum_{\ell_2=\ell_1+2}^\infty\pi(2^{\ell_2}R;2^{\ell_1}R)\ldots\!\!\sum_{\ell_n=\ell_{n-1}+2}^\infty\pi(2^{\ell_n}R;2^{\ell_{n-1}}R)\\
\times\sum_{\ell_{n+1}=\ell_n+2}^\infty\expec{\Big(\oscd{2^{\ell_{n+1}}R}Z(A)\Big)^2\mathds1_{2^{\ell_{n+1}-1}R\le\rho(2^{\ell_n}R)<2^{\ell_{n+1}} R}}.
\end{multline*}
With the choice~\eqref{eq:choice-R2} of $R$ in form of
\begin{align*}
\sup_{\ell_0\ge0}\sum_{\ell=\ell_0+2}^\infty\pi(2^\ell R;2^{\ell_0}R)=\sup_{\ell_0\ge0}\prm{\rho(2^{\ell_0}R)\ge2^{\ell_0+1}R}\le\tfrac14,
\end{align*}
setting $\tilde\pi(\ell):=\sup_{\ell_0:0\le\ell_0\le\frac14\ell}\pi(\ell;\ell_0)$, and recalling that the random variable $Z(A)$ is bounded, we deduce
\begin{multline*}
\expec{\big(Z(A)-Z(A(\X^x))\big)^2}\le\expec{\Big(\oscd{R}Z(A)\Big)^2}\\
+2\bigg(\sum_{m=0}^{n-1}2^{-m}\bigg)\sum_{\ell=1}^\infty\tilde\pi(2^{\ell} R)\,\expec{\Big(\oscd{2^{\ell+1}R}Z(A)\Big)^2}+2^{-n-2}\|Z\|_{\Ld^\infty}.
\end{multline*}
Letting $n\uparrow\infty$, we thus obtain
\begin{multline*}
\expec{\big(Z(A)-Z(A(\X^x))\big)^2}\le\expec{\Big(\oscd{R}Z(A)\Big)^2}+4\sum_{\ell=1}^\infty\tilde\pi(2^{\ell} R)\,\expec{\Big(\oscd{2^{\ell+1}R}Z(A)\Big)^2}.
\end{multline*}
{\color{black}Noting that by definition $\sup_{\frac12\ell\le r\le\ell}\tilde\pi(r)\le \pi_0(\ell)$, bounding sums by integrals, and using the definition of $\pi$, the conclusion~\eqref{eq:AG-centered} follows.}
\end{proof}


\section{Examples}\label{sec:examples}

In this section, we consider four main representative examples: Gaussian fields, random tessellations associated with a Poisson process, random parking bounded inclusions, and Poisson inclusions with unbounded random radii. The main results are summarized in the table below.

\bigskip

\begin{tabular}{|M{3.6cm}|M{5.3cm}|M{5cm}|}
\hline
{\bf Example of field} & {\bf Key property} & {\bf Functional inequalities} \\
\hline
Gaussian field & covariance function $\Cc$ $\sup_{B(x)}|\Cc|\le c(|x|)$ & ($\parfct{}$-MSG), ($\parfct{}$-MLSI) weight $\pi(\ell)\simeq (-c'(\ell))_+$\tabularnewline
\hline
Poisson tessellations & $\sigma(\X)$-measurable action radius & ($\osc$-MSG), ($\osc$-MLSI) weight $\pi(\ell)\simeq e^{-\frac1C\ell^d}$\\
\hline
Random parking bounded inclusions & $\sigma(\X)$-measurable action radius \& exponential stabilization & ($\osc$-MSG), ($\osc$-MLSI) weight $\pi(\ell)\simeq e^{-\frac1C\ell}$\\
\hline
Poisson inclusions \qquad\qquad with random radii & radius law $V$ $\gamma(\ell):=\pr{\ell\le V<\ell+1}$ & ($\osc$-MSG) \qquad\qquad\qquad weight $\pi(\ell)\simeq (\ell+1)^d\gamma(\ell)$ (and ($\osc$-LSI) if $V$ bounded)\\
\hline
\end{tabular}
\medskip

\subsection{Gaussian random fields}\label{chap:gaussian}

Gaussian random fields are the main examples of deterministically localized fields as introduced in Section~\ref{chap:condunif-pasunif}.
The following result is a multiscale weighted reformulation of the Malliavin-Poin\-car\'e inequality in \cite{Houdre-PerezAbreu-95} (see also the ``coarsened'' functional inequalities used in the first version of~\cite{GNO-reg} for Gaussian fields).
As shown in the companion article~\cite[Proposition~2.3]{DG1}, this result is sharp: each sufficient condition is (essentially) necessary.
The proof is postponed to Appendix~\ref{chap:condunif}.

\begin{theor}\label{cor:gaussiansg}
Let $A$ be a jointly measurable stationary Gaussian random field on $\R^d$ with covariance function $\Cc(x):=\cov{A(x)}{A(0)}$.
\begin{enumerate}[(i)]
\item If $x\mapsto \sup_{B(x)}|\Cc|$ is integrable, then $A$ satisfies \emph{($\parfct{}$-SG)} and \emph{($\parfct{}$-LSI)} with any radius $R>0$.
\item If $\sup_{B(x)}|\Cc| \le c(|x|)$ holds for some non-increasing Lipschitz function $c:\R_+\to\R_+$, then $A$ satisfies \emph{($\parfct{}$-MSG)} and \emph{($\parfct{}$-MLSI)} with weight $\pi(\ell)\simeq |c'(\ell)|$.
\item If $\F\Cc\in\Ld^1(\R^d)$ and if $\sup_{B(x)}|\F^{-1}(\sqrt{\F\Cc})|\le r(|x|)$ holds for some non-increasing Lipschitz function $r:\R_+\to\R_+$, then $A$ satisfies \emph{($\parfct{}$-MCI)} with weight $\pi(\ell)\simeq (\ell+1)^d\,r(\ell)|r'(\ell)|$.
\qedhere
\end{enumerate}
\end{theor}

\subsection{Poisson random tessellations}\label{chap:tessel}
In this section, we consider random fields that take i.i.d.\@ values on the cells of a tessellation associated with a stationary random point process $\Pc$ on $\R^d$. 
Such random fields can be formalized as projections of decorated random point processes. Given a point process $\Pc$ on $\R^d$ and given a random element $G$ with values in some measurable space $X$, we call {\it decorated random point process associated with $\Pc$ and $G$} a point process $\hat\Pc$ on $\R^d\times X$ defined as follows: choose a measurable enumeration $\Pc=\{P_j\}_j$,  pick independently a sequence $(G_j)_j$ of i.i.d.\@ copies of the random element $G$, and set $\hat\Pc:=\{P_j,G_j\}_j$
(that is, in measure notation, $\hat\Pc:=\sum_j\delta_{(P_j,G_j)}$).
By definition, $\hat\Pc$ is completely independent whenever $\Pc$ is.

\medskip

We focus here on the case when the underlying point process $\Pc$ is
some Poisson point process $\Pc=\Pc_0$ on $\R^d$ with intensity $\mu=1$. Choose a measurable random field $V$ on $\R^d$, corresponding to the values on the cells.
We study both Voronoi and Delaunay tessellations.
\begin{itemize}
\item{\it Voronoi tessellation:}
Let $\hat\Pc_1:=\{P_j,V_j\}_j$ denote a decorated point process associated with the random point process $\Pc_0:=\{P_j\}_j$ and the random element $V$ (hence $(V_j)_j$ is a sequence of i.i.d.\@ copies of the random field $V$).
We define a $\sigma(\hat\Pc_1)$-measurable random field $A_1$ as follows,
\[A_1(x)=\sum_jV_j(x)\mathds1_{C_j}(x),\]
where $\{C_j\}_j$ denotes the partition of $\R^d$ into the Voronoi cells associated with the Poisson points $\{P_j\}_j$, that is,
\[C_j:=\{x\in\R^d:|x-P_j|<|x-P_k|,~\forall k\ne j\}.\]
\item{\it Delaunay tessellation:}
Let $\tilde V:=(\tilde V_\zeta)_\zeta$ denote a family of i.i.d.\@ copies of the random element $V$, indexed by sets $\zeta$ of $d+1$ distinct integers. 
We define a random field $A_2$ as follows,
\[A_2(x)=\sum_j \tilde V_{\zeta(D_j)}(x)\mathds1_{D_j}(x),\]
where $\{D_j\}_j$ denotes the partition of $\R^d$ into the Delaunay $d$-simplices associated with the Poisson points $\{P_j\}_j$ (the Delaunay triangulation is indeed almost surely uniquely defined), and where $\zeta(D_j)$ denotes the set of the $d+1$ indices $i_1,\dots,i_{d+1}$ of the vertices $P_{i_1},\dots,P_{i_{d+1}}$ of $D_j$.
\end{itemize}

Since large holes in the Poisson process have exponentially small probability, large cells in the corresponding Voronoi or Delaunay tessellations also have exponentially small probability. This allows one to prove the following multiscale functional inequalities with stretched exponential weights.
\begin{prop}\label{prop:voronoi}
For $s=1,2$, the above-defined random field $A_s$ satisfies~\emph{($\osc$-MSG)}, \emph{($\osc$-MLSI)}, and \emph{($\osc$-MCI)}
with weight $\pi(\ell)=Ce^{-\frac1C\ell^d}$.
\end{prop}
\begin{proof}
We focus on the case of the Voronoi tessellation (the argument for the Delaunay tessellation is similar).
{\color{black}While Theorem~\ref{th:ar} does not apply to this setting (the independence assumption~(c) is not satisfied), we may appeal to Theorem~\ref{th:ar-rpm}. We need to construct and control action radii, which we do in two separate steps.}

\medskip
\step1 Definition and properties of the action radius.

\noindent
Let $x\in\R^d$, $\ell \in \N$ be fixed. Changing the point configuration of $\hat\Pc_1=\{P_j,V_j\}_j$ inside $Q_{2\ell+1}(x)\times\R^{\R^d}$ only modifies the Voronoi tessellation (hence the field $A_1$) inside the set
\begin{multline*}
G_{\Pc_0,\ell}(x)\,:=\,\big\{y\in\R^d:\exists z\in Q_{2\ell+1}(x)\\\text{ such that $|y-z|\le|y-X|$ for all $X\in \Pc_0\setminus  Q_{2\ell+1}(x)$}\big\}.
\end{multline*}
{\color{black}Note that $G_{\Pc_0,\ell}(x)$ is a simply connected closed set and contains $Q_{2\ell+1}(x)$.
An action radius for $A_1$ wrt $\hat\Pc_1$ on $Q_{2\ell+1}(x)\times\R^{\R^d}$ is then given for instance by
\[\inf\{\rho>0:Q_{2\ell+1}(x)+B_\rho\supset G_{\Pc_0,\ell}(x)\}=\max_{v\in \partial G_{\Pc_0,\ell}(x)}d(v,Q_{2\ell+1}(x)),\]
but in view of the measurability property~(c) we rather make the following weaker choice,
\[\rho_{x}^{\ell}:= 1+2\max_{v\in \partial G_{\Pc_0,\ell}(x)}d(v,Q_{2\ell+1}(x)).\]}
Property~(a) of Theorem~\ref{th:ar-rpm} is then proved, and the stationarity property~(b) follows by construction.

\smallskip\noindent
{\color{black}Next, we establish the measurability property~(c) of Theorem~\ref{th:ar-rpm}, that is, we prove that $\rho_x^\ell$ is $\sigma(\Pc_0|_{Q_{2(\ell+\rho_x^\ell)+1}(x)\setminus Q_{2\ell+1}(x)})$-measurable.
Since $\rho_x^\ell$ is $\sigma(\Pc_0|_{\R^d\setminus Q_{2\ell+1}(x)})$-measurable by construction, it remains to prove it is $\sigma(\Pc_0|_{Q_{2(\ell+\rho_x^\ell)+1}(x)})$-measurable.
To this aim, let $\tilde \Pc$ be an arbitrary locally finite point set and consider the compound point set
\[\tilde \Pc_{0,\ell}(x)\:=\,\Pc_0|_{Q_{2(\ell+\rho_x^\ell)+1}(x)}   \cup \tilde \Pc|_{\R^d  \setminus Q_{2(\ell+\rho_x^\ell)+1}(x)}.\]
The claimed measurability then follows from the identity
\begin{equation}\label{eq:ident-cruc}
G_{\tilde \Pc_{0,\ell}(x),\ell}(x)=G_{\Pc_0,\ell}(x),
\end{equation}
as this indeed implies that for all $r>0$ the event $\{\rho_x^\ell<r\}$ coincides with
\[\{2\diam G_{\Pc_0\cap Q_{2(\ell+r)+1}(x)}+1-\ell<r\}\in\sigma(\Pc_0|_{Q_{2(\ell+r)+1}(x)}).\]
It remains to establish~\eqref{eq:ident-cruc}.
Consider  $y\in (G_{\Pc_0,\ell}(x)+\frac14B)\setminus G_{\Pc_0,\ell}(x)$ (the $\frac14$-fattened boundary of $G_{\Pc_0,\ell}(x)$). Since $y \notin G_{\Pc_0,\ell}(x)$, there exists $X \in \Pc_0\setminus Q_{2\ell+1}(x)$ such that $|y-X|<|y-z|$ holds for all $z\in Q_{2\ell+1}(x)$. The triangle inequality then yields
\begin{multline*}
|X-x|_\infty\le |X-y|+|y-x|_\infty<d(y,Q_{2\ell+1}(x))+|y-x|_\infty\\
\le \ell+\tfrac12+2d(y,Q_{2\ell+1}(x))
\le \ell+1+2\max_{v\in \partial G_{\Pc_0,\ell}(x)}d(v,Q_{2\ell+1}(x))=\ell+\rho_x^\ell,
\end{multline*}
that is, $X\in Q_{2(\ell+\rho_x^\ell)+1}(x)$, hence $X\in\tilde\Pc_{0,\ell}(x)$, which in turn implies $y\notin G_{\tilde\Pc_{0,\ell}(x),\ell}(x)$.
This proves the inclusion $\partial G_{\Pc_0,\ell}(x)\subset\R^d\setminus G_{\tilde\Pc_{0,\ell}(x),\ell}(x)$. Conversely, the same argument yields $\partial G_{\tilde\Pc_{0,\ell}(x),\ell}(x)\subset\R^d\setminus G_{\Pc_{0},\ell}(x)$. Since $G_{\Pc_{0},\ell}(x)$ and $G_{\tilde\Pc_{0,\ell}(x),\ell}(x)$ are simply connected closed sets, the identity~\eqref{eq:ident-cruc} follows, thus proving the measurability property~(c).}

\medskip
\step2 Control of the weight.

\noindent
In view of Step~1, we may apply Theorem~\ref{th:ar-rpm} and it remains to estimate the weights.
By scaling, it is enough to consider $\ell=0$ (we omit the subscripts~$\ell$ in the notation) and a Poisson point process $\Pc_0^\mu$ of general intensity $\mu>0$.
Denote by $\mathcal C_i=\{x\in \R^d:x_i\ge \frac56 |x|\}$ the $d$ cones in the canonical directions $e_i$ of $\R^d$, and consider the $2d$ cones $\mathcal C_i^\pm:=\pm (2e_i+C_i)$.
By an elementary geometric argument, for some constant $C\simeq1$ the following implication holds: for all $L>C$,
$$
\sharp \big(\Pc_0^\mu \cap \mathcal C_i^\pm \cap \{x:C\le |x_i|\le 2L\}\big)>0 \text{ for all }i \text{ and }\pm \quad \implies \quad  \diam G_{\Pc_0^\mu}(0)\le CL.  
$$
{\color{black}A union bound then yields for all $L>C^2$,
\begin{eqnarray*}
\pr{ \diam G_{\Pc_0^\mu}(0) \ge L} &\le&\pr{\exists i, \pm: \sharp \big(\Pc_0^\mu \cap \mathcal C_i^\pm \cap \{x:C\le |x_i|\le\tfrac 2C L\}\big)=0}\\
&\le& 2d\,e^{-\mu (\frac LC)^d}.
\end{eqnarray*}
By scaling, as the intensity of the Poisson process scales like the volume, and recalling that~$\Pc_0$ is chosen with unit intensity, we deduce for all $\ell\ge0$ and $L>C^2$,
\[\pr{ \diam G_{\Pc_0,\ell}(0) \ge L}\le2d\,e^{-(\frac LC)^d}.\]
Noting that the definition of the action radius in Step~1 yields
\[\rho_0^{\ell}:= 1+2\max_{v\in \partial G_{\Pc_0,\ell}(x)}d(v,Q_{2\ell+1}(x))\le2\diam G_{\Pc_0,\ell}(0)-4\ell,\]
we deduce $\pr{\rho_0^\ell\ge L}\le \pr{2\diam G_{\Pc_0,\ell}(0)\ge L+4\ell}\le Ce^{-\frac1C(L+\ell)^d}$ for all $\ell,L\ge0$, and the conclusion follows.}
\end{proof}

\subsection{Random parking process}\label{chap:RPM}
{\color{black}In this section we let $\Pc$ denote the random parking point process on $\R^d$ with hardcore radius $R>0$.}
As shown by Penrose~\cite{Penrose-01} (see also~\cite[Section~2.1]{Gloria-Penrose-13}), it can be constructed as a transformation $\Pc=\Phi(\Pc_0)$ of a Poisson point process $\Pc_0$ on $\R^d\times\R_+$ with intensity~$1$. Let us recall the graphical construction of this transformation $\Phi$. We first construct an oriented graph on the points of $\Pc_0$ in $\R^d\times\R_+$, by putting an oriented edge from $(x,t)$ to $(x',t')$ whenever $B(x,R)\cap B(x',R)\ne\varnothing$ and $t<t'$ (or $t=t'$ and $x$ precedes $x'$ in the lexicographic order, say). We say that $(x',t')$ is an offspring (resp.\@ a descendant) of $(x,t)$, if $(x,t)$ is a direct ancestor (resp.\@ an ancestor) of $(x',t')$, that is, if there is an edge (resp.\@ a directed path) from $(x,t)$ to $(x',t')$. The set $\Pc:=\Phi(\Pc_0)$ is then constructed as follows. Let $F_1$ be the set of all roots in the oriented graph (that is, the points of $\Pc_0$ without ancestor), let $G_1$ be the set of points of $\Pc_0$ that are offsprings of points of $F_1$, and let $H_1:=F_1\cup G_1$. Now consider the oriented graph induced on $\Pc_0\setminus H_1$, and define $F_2,G_2,H_2$ in the same way, and so on. By construction, the sets $(F_j)_j$ and $(G_j)_j$ are all disjoint and constitute a partition of $\Pc_0$. We finally define $\Pc:=\Phi(\Pc_0):=\bigcup_{j=1}^\infty F_j$.

\medskip
{\color{black}In this setting, in view of the exponential stabilization results of~\cite{Schreiber-Penrose-Yukich-07}, we show that there exists an action radius with exponential moments for $\Pc$ wrt $\Pc_0$, leading to the following multiscale functional inequalities with exponential weights.

\begin{prop}\label{prop:rpm}
The above-defined random parking point process $\Pc$ with hardcore radius $R=1$ satisfies
\emph{($\osc$-MSG)}, \emph{($\osc$-MLSI)}, and~\emph{($\osc$-MCI)}
with weight $\pi(\ell)=Ce^{-\frac1{C}\ell}$.
\end{prop}}
\begin{proof}
{\color{black}
The independence assumption~(c) of Theorem~\ref{th:ar} is not satisfied and we rather appeal to Theorem~\ref{th:ar-rpm}.
In order to construct action radii, we rely on the notion of {\it causal chains} defined in the proof of~\cite[Lemma~3.5]{Schreiber-Penrose-Yukich-07}, to which we refer the reader.
Note that for all consecutive points $(x,t)$ and $(y,s)$ in a causal chain we necessarily have $|x-y|<2$ and $t<s$.
By definition, an action radius for $\Pc$ wrt $\Pc_0$ on $Q_{2\ell+1}(x)\times\R_+$ can be defined as the supremum of the distances $2+d(y,Q_{2\ell+1}(x))$ on the set of points $(y,s)\in\Pc_0$ such that there exists a causal chain from a point of $\Pc_0$ in $((Q_{2\ell+1}(x)+B_{2})\setminus Q_{2\ell+1}(x))\times\R_+$ towards $(y,s)$. We denote by $\rho_x^\ell$ this maximum. By construction, we note that this random variable $\rho_x^\ell$ is $\sigma\big(\Pc_0|_{((Q_{2\ell+1}(x)+B_{\rho_x^\ell})\setminus Q_{2\ell+1}(x))\times\R_+}\big)$-measurable.

\smallskip\noindent
It remains to estimate the decay of its probability law.
First, note that by definition the event $\rho_x^\ell>L$ entails the existence of some $(y,s)\in \Pc_0$ with $y\in(Q_{2\ell+1}(x)+B_{L+2})\setminus(Q_{2\ell+1}(x)+B_{L})$ and of a causal chain from a point of $\Pc_0$ in $((Q_{2\ell+1}(x)+B_{2})\setminus Q_{2\ell+1}(x))\times\R_+$ towards $(y,s)$.
Second, the exponential stabilization result of~\cite[Lemma~3.5]{Schreiber-Penrose-Yukich-07} states that for all $z\in\R^d$ and all $L>0$ the probability that there exists $(y,s)\in Q(z)\times\R_+$ and a causal chain from $(y,s)$ towards a point outside $(Q(z)+B_{L})\times\R_+$ is bounded by $Ce^{-\frac1{C}L}$.
For $L\ge R$, covering $(Q_{2\ell+1}(x)+B_{L+2})\setminus(Q_{2\ell+1}(x)+B_{L})$ with $C(L+\ell+1)^{d-1}$ unit cubes and covering $(Q_{2\ell+1}(x)+B_{2})\setminus Q_{2\ell+1}(x)$ with $C(\ell+1)^{d-1}$ unit cubes,
a union bound then yields
\begin{equation}\label{eq:bound-rhoxl-rpm}
\prm{\rho_x^\ell>L}\le C(L+\ell+1)^{d-1}(\ell+1)^{d-1}e^{-\frac1{C}L}\le C(\ell+1)^{2(d-1)}e^{-\frac1{C}L}.
\end{equation}
All the assumptions of Theorem~\ref{th:ar-rpm} are then satisfied with $\pi(\ell)=Ce^{-\frac1{C}\ell}$, and the conclusion follows.}
\end{proof}

{\color{black}
\begin{rem}\label{lem:general-hardcore/decimated}
We conclude this section with a remark on the following two extensions: we analyze the dependence on a general hardcore parameter $R>0$, and we consider Bernoulli modifications to generate a hardcore point process with arbitrary intensity.
\begin{enumerate}[(a)]
\item Let $\Pc=\{P_j\}_j$ be a random point process on $\R^d$ that satisfies {($\osc$-MSG)}, {($\osc$-MLSI)}, and~{($\osc$-MCI)} with weight $\pi$.
Then for all $R>0$, the dilated process $\Pc_R:=\{RP_j\}_j$ satisfies~{($\osc$-MSG)}, {($\osc$-MLSI)}, and~{($\osc$-MCI)} with weight 
$$
\pi_R(\ell):= R^{-1} \big(\tfrac{\ell+1} {\ell+ R}\big)^{d}\pi(\tfrac \ell R).
$$
In addition, if  $\Pc$  is hardcore with parameter 1, then $\Pc_R$ is hardcore with parameter~$R$.
(If $\Pc$ is the random parking point process with hardcore radius $1$, then the dilated process $\Pc_R$ coincides in law with the random parking point process with radius~$R$.)
Denoting by $D_R$ the dilation by $R$, and by $\varR{\cdot}$ and $\expecR{\cdot}$ the variance and expectation with respect to $\Pc_R$, the claim simply follows from a change of variables,
\begin{eqnarray*}
\qquad\varR{Z}&=&\var{Z\circ D_R}\\
&\le& \expec{\int_0^\infty\int_{\R^d} \Big(\oscd{\Pc,B_\ell(x)}Z\circ D_R\Big)^2dx\,(\ell+1)^{-d} \pi(\ell)\,d\ell}
\\
&=&\expecR{\int_0^\infty\int_{\R^d} \Big(\oscd{\Pc_R,B_{R\ell}(Rx)}Z\Big)^2dx\, (\ell+1)^{-d} \pi(\ell)\,d\ell}
\\
&=&R^{-d-1}\,\expecR{\int_0^\infty\int_{\R^d} \Big(\oscd{\Pc_R,B_{\ell}(x)}Z\Big)^2dx\, (\tfrac\ell R+1)^{-d} \pi(\tfrac\ell R)\,d\ell}.
\end{eqnarray*}

\item A simple way to modify the intensity of the random parking point process $\Pc=\{P_j\}_j$ consists in defining for $0\le\lambda \le1$ the corresponding $\lambda$-decimated process
\[\qquad\Pc^\lambda:=\{P_j\in\Pc:b_j=1\},\]
where $\{b_j\}_j$ is an i.i.d.\@ sequence of Bernoulli random variables with $\pr{b_j=1}=\lambda$, independent of $\Pc$.
Alternatively, since the hardcore condition ensures that points of $\Pc$ are always at distance $>2$ from one another,
we can rather describe the law of $\Pc^\lambda$ via
\[\qquad\Pc^\lambda=\Big\{P_j:\exists z\in\tfrac2{\sqrt d}\Z^d,\, P_j \in\Pc \cap Q_{\frac2{\sqrt{d}}}(z)\text{ and }b_z=1\Big\},\]
where $\{b_z\}_z$ is an i.i.d.\@ sequence of Bernoulli random variables with $\pr{b_z=1}=\lambda$.
This point process $\Pc^\lambda$ is again stationary and ergodic.
Denoting by $\rho_x^\ell$ an action radius for the random parking point process $\Pc$ wrt $\Pc_0$ on $Q_{2\ell+1}(x)\times\R_+$,
an action radius for $\Pc^\lambda$ wrt $\Pc_0\times\{b_z\}_z$ on $Q_{2\ell+1}(x)\times\R_+$ is given by
$$
\qquad\rho_{\lambda,x}^\ell := \sup\Big\{0\le r\le \rho_x^\ell \,:\, \exists z\in \tfrac2{\sqrt d} \Z^d,\,b_z=1\text{  and  }Q_{\frac2{\sqrt{d}}}(z)\cap\partial(Q_{2\ell+1}(x)+B_r)\ne\varnothing\Big\}
$$
In this case, in view of~\eqref{eq:bound-rhoxl-rpm},
\[\prm{\rho_{\lambda,x}^\ell>L}\,\le\,C\lambda(L+\ell+1)^{d-1}\prm{\rho_{x}^\ell>L}\,\le\,C\lambda(\ell+1)^{3(d-1)}e^{-\frac1CL},\]
hence by Theorem~\ref{th:ar-rpm} the decimated process $\Pc^\lambda$ satisfies {($\osc$-MSG)}, {($\osc$-MLSI)}, and~{($\osc$-MCI)} with weight $\pi_\lambda(\ell)=C\lambda e^{-\frac1{C}\ell}$, that is, a prefactor $\lambda$ is gained.
\qedhere
\end{enumerate}
\end{rem}}

\subsection{Hardcore Poisson process}\label{sec:hardcore}
{\color{black}
In this section we consider the hardcore Poisson point process $\Pc$ on $\R^d$ with parameters $R,\lambda$, which we define via Penrose's graphical construction $\Pc=\Pc(\Pc_0)$ recalled in Section~\ref{chap:RPM} with hardcore radius $R$ and starting from a Poisson point process $\Pc_0$ of intensity $\lambda$ on $\R^d\times [0,1]$ (instead of a Poisson process on the whole of $\R^d\times\R^+$ as for the random parking process).
The so-defined point process $\Pc$ is stationary, ergodic, and has intensity $\lambda(1+O(\lambda R^d))$. Points of $\Pc$ are always at distance $>2R$ from each other as for the random parking process, but it is not jammed in the sense that arbitrarily large empty spaces still appear as for the Poisson process.
In this setting, we establish the following multiscale functional inequalities with Poisson weights.

\begin{prop}\label{prop:SG0}
Provided that $\lambda R^d\le1$,
the above-defined hardcore Poisson process~$\Pc$ with parameters $R,\lambda$ satisfies \emph{($\osc$-MSG)}, \emph{($\osc$-MLSI)}, and \emph{($\osc$-MCI)} with weight $\pi(\ell)=C\lambda R^{-1}(R+1)^{d}  e^{-\frac{\ell}{CR}\log\frac\ell{CR}}$.
\end{prop}

%

\begin{proof}
By Remark~\ref{lem:general-hardcore/decimated}(a), it suffices to argue for hardcore radius $R=1$. By this rescaling, the Poisson point process on $\R^d \times [0,1]$ in the graphical construction now has intensity $\lambda R^d \le1$, and can be seen as the $\lambda R^d$-decimation of a Poisson point process with unit intensity, as in Remark~\ref{lem:general-hardcore/decimated}(b).
It is thus enough to treat the case $R=\lambda=1$. The proof is again an application of Theorem~\ref{th:ar-rpm}.
We start with the construction of an action radius $\rho_{x}^\ell$ for $\Pc$ wrt $\Pc_0$ on $Q_{2\ell+1}(x)\times[0,1]$ for all $x,\ell$.
We define causal chains as sequences $\{(y_j,s_j)\}_{j=1}^n$ such that $|y_j-y_{j+1}|<2$ and $s_j<s_{j+1}$.
The action radius $\rho_{x}^\ell$ can then be chosen as the maximum of the distances $2+d(y,Q_{2\ell+1}(x))$ on the set of points $(y,s)\in\Pc_0$ such that there exists a causal chain from a point of $\Pc_0$ in $((Q_{2\ell+1}(x)+B_{2})\setminus Q_{2\ell+1}(x))\times[0,1]$ towards $(y,s)$.
By construction, we note that this random variable $\rho_x^\ell$ is $\sigma\big(\Pc_0|_{((Q_{2\ell+1}(x)+B_{\rho_x^\ell})\setminus Q_{2\ell+1}(x))\times\R_+}\big)$-measurable.

\smallskip\noindent
It remains to estimate the decay of the probability law of the action radii. The event $\rho_x^\ell>L$ entails the existence of some $(y,s)\in \Pc_0$ with $y\in(Q_{2\ell+1}(x)+B_{L+2})\setminus(Q_{2\ell+1}(x)+B_{L})$ and of a causal chain from a point of $((Q_{2\ell+1}(x)+B_{2})\setminus Q_{2\ell+1}(x))\times\R_+$ towards $(y,s)$.
Arguing as in~\cite[proof of~(0.2) in Lemma~4.2]{Penrose-Yukich-02}, for all $\theta>0$, the probability that there exists a causal chain from a point of $\Pc_0$ in $Q(x)\times[0,1]$ to a point of $\Pc_0$ in $Q(y)\times[0,1]$ is bounded by
\[e^\theta\big(\tfrac{3^dC}{C+\theta}\big)^{|x-y|},\]
that is, after optimization in $\theta$,
\[Ce^{-|x-y|\log\frac{|x-y|}{C}}.\]
By a similar covering argument as in the proof of Proposition~\ref{prop:rpm}, all the assumptions of Theorem~\ref{th:ar-rpm} are then satisfied with $\pi(\ell)=Ce^{-\frac\ell C\log\frac\ell C}$, and the conclusion follows.
\end{proof}

}

\subsection{Random inclusions with random radii}\label{chap:genincl}

We consider typical examples of random fields on $\R^d$ taking random values on random inclusions centered at the points of some random point process $\Pc$. The inclusions are allowed to have i.i.d.\@ random shapes (hence in particular i.i.d.\@ random radii). For the random point process $\Pc$,  we consider  projections $\Phi(\Pc_0)$ of some Poisson point process $\Pc_0$ on $\R^d\times\R^l$ with intensity $\mu>0$, and shall assume that for all $x\in\Z^d$ the process $\Pc$ admits an action radius $\rho_x$ wrt $\Pc_0$ on $Q(x)\times\R^l$.

\medskip

We turn to the construction of the random inclusions.
Let $V$ be a nonnegative random variable (corresponding to the random radius of the inclusions).
In order to define the random shapes, we consider the set $Y$ of all nonempty Borel subsets $E\subset\R^d$ with $\sup_{x\in E}|x|=1$, and endow it with the $\sigma$-algebra $\Y$ generated by all subsets of the form $\{E\in Y\,:\, x_0\in E\}$ with $x_0\in\R^d$.
Let $S$ be a random
nonempty Borel subset of $\R^d$ with $\sup_{x\in S}|x|=1$ a.s., that is, a random element in the measurable space $Y$.
(Note that $V$ and $S$ need not be independent.)
Let $\hat\Pc_0:=\{P_j,V_j,S_j\}_j$ be a decorated point process associated with the random point process $\Pc_0=\{P_j\}_j$ and the random element $(V,S)$.
The collection of random inclusions is then given by $\{I_j\}_j$ with $I_j:=P_j+V_jS_j$.

\medskip

It remains to associate random values to the random inclusions. Since inclusions may intersect each other, several constructions can be considered; we focus on the following three typical choices.
\begin{itemize}
\item Given $\alpha,\beta\in\R$, we set $\hat\Pc_1:=\hat\Pc_0$, and we consider the $\sigma(\hat\Pc_1)$-measurable random field $A_1$ that is equal to $\alpha$ inside the inclusions, and to $\beta$ outside. More precisely,
\[A_1:=\beta+(\alpha-\beta)\mathds1_{\bigcup_jI_j}.\]
The simplest example is the random field $A_1$ obtained for $\Pc$ a Poisson point process on $\R^d$ with intensity $\mu=1$, and for $S$ the unit ball centered at the origin in $\R^d$; this is referred to as the {\it Poisson unbounded spherical inclusion model}.
\smallskip\item Let $\beta\in\R$, let $f:\R\to\R$ be a Borel function, and let $W$ be a measurable random field on $\R^d$. Let $\hat\Pc_2:=\{P_j,V_j,S_j,W_j\}$ be a decorated point process associated with $\hat\Pc_0$ and $W$. We then consider the $\sigma(\hat\Pc_2)$-measurable random field $A_2$ that is equal to $f(\sum_{j:x\in I_j}W_j)$ at any point $x$ of the inclusions, and to $\beta$ outside.
More precisely,
\[A_2(x):=\beta+\bigg(f\Big(\sum_{j}W_j(x)\mathds1_{I_j}(x)\Big)-\beta\bigg)\mathds1_{\bigcup_jI_j}(x).\]
(Of course, this example can be generalized by considering more general functions than simple sums of the values $W_j$; the corresponding concentration properties will then remain the same.)
\smallskip\item Let $\beta\in\R$, let $W$ be a measurable random field on $\R^d$, and let $U$ denote a uniform random variable on $[0,1]$. Let $\hat\Pc_3:=\{P_j,V_j,S_j,W_j,U_j\}$ be a decorated point process associated with $\hat\Pc_0$ and $(W,U)$. Given a $\sigma(VS,W)$-measurable random variable $P(VS,W)$, we say that inclusion $I_j$ has the priority on inclusion $I_i$ if $P(V_jS_j,W_j)<P(V_iS_i,W_i)$ or if $P(V_jS_j,W_j)=P(V_iS_i,W_i)$ and $U_j<U_i$. Since the random variables $\{U_j\}_j$ are a.s.\@ all distinct, this defines a priority order on the inclusions on a set of maximal probability. Let us then relabel the inclusions and values $\{(I_j,V_j)\}_j$ into a sequence $(I_j',V_j')_j$ in such a way that for all $j$ the inclusion $I_j'$ has the $j$-th highest priority.
We then consider the $\sigma(\hat\Pc_3)$-measurable random field $A_3$ defined as follows,
\[A_3:=\beta+\sum_j(W_j'-\beta)\mathds1_{I_j'\setminus\bigcup_{i:i<j}I_i'}.\]
(Note that this example includes in particular the case when the priority order is purely random (choosing $P\equiv0$), as well as the case when the priority is given to inclusions with e.g.\@ larger or smaller radius (choosing $P(VS,W)=V$ or $-V$, respectively).)
\end{itemize}
In each of these three examples, $s=1,2,3$, the random field $A_s$ is $\sigma(\hat\Pc_s)$-measurable, for some completely independent random point process $\hat\Pc_s$ on $\R^d\times \R^l\times\R_+\times Y_s$ and some measurable space $Y_s$ (the set $\R^d\times\R^l$ stands for the domain of the point process $\Pc_0=\{P_j\}_j$, and the set $\R_+$ stands for the domain of the radius variables~$\{V_j\}_j$).
In order to recast this into the framework of Section~\ref{eq:transfiid}, we may define $\X_s(x,t,v):=\Pc_s|_{Q(x)\times Q(t)\times Q(v)\times Y_s}$, so that $\X_s$ is a completely independent measurable random field on the space $X=\Z^d\times \Z^l\times\Z$ with values in the space of (locally finite) measures on $Q^d\times Q^l\times Q^1\times Y_s$.

\medskip

{\color{black}Rather than stating a general result, we focus on the typical examples of the Poisson point process and of the random parking or hardcore Poisson processes.} For the latter, a refined analysis is needed to avoid a loss of integrability.
{\color{black}Note that logarithmic Sobolev inequalities are only obtained in case of bounded radii; this is due to the strong additional condition for the validity of~\eqref{eq:arlsi2} in Theorem~\ref{th:ar}.}
The proof below yields slightly more general results than stated and can easily be adapted to various other situations.
{\color{black}\begin{prop}\label{prop:genincl-loss}
Set $\gamma(v):=\p[v-1/2\le V< v+1/2]$ and $\tilde\gamma(v):=\sup_{u\ge v}\gamma(u)$.
\begin{enumerate}[(i)]
\item Assume that $\Pc=\Pc_0$ is a Poisson point process on $\R^d$ with constant intensity $\mu$.
Then, for each $s=1,2,3$, the above-defined random field $A_s$ satisfies~\emph{($\osc$-MSG)} and~\emph{($\osc$-MCI)}
with weight $\pi(\ell)=\mu\,(\ell+1)^d\tilde\gamma(\frac1{\sqrt d}\ell-3)$.
If the radius law $V$ is uniformly bounded, the standard logarithmic Sobolev inequality \emph{($\osc$-LSI)} further holds.
\smallskip\item Assume that $\Pc$ is a random parking process on $\R^d$ as constructed in Section~\ref{chap:RPM}.
Then, for each $s=1,2,3$, the above-defined random field $A_s$ satisfies \emph{($\osc$-MSG)} with weight $\pi(\ell)=C\big(e^{-\ell/C}+(\ell+1)^d\tilde\gamma(\frac14\ell-1)\big)$.
If the radius law $V$ is uniformly bounded,
the logarithmic Sobolev inequality \emph{($\osc$-MLSI)} further holds with weight $Ce^{-\ell/C}$. If $\Pc$ is rather the hardcore Poisson process on $\R^d$ as constructed in Section~\ref{sec:hardcore}, then the same result holds with $e^{-\ell/C}$ replaced by $e^{-\frac\ell C\log\frac\ell C}$.
\qedhere
\end{enumerate}
\end{prop}}

{\color{black}\begin{rem}
As shown in the proof, in the case of item~(ii), a corresponding covariance inequality holds next to ($\osc$-MSG) in the following form, for all $\sigma(A_s)$-measurable random variables $Y(A_s),Z(A_s)$,
\begin{multline}\label{eq:MSG-RPM}
\qquad\cov{Y(A_s)}{Z(A_s)}\le \int_{\R^d}\bigg(\int_0^\infty\expec{\Big(\oscd{A_s,B_{2\ell+1}(x)}Y(A_s)\Big)^2}(\ell+1)^{-d}\pi(\ell)\,d\ell\bigg)^\frac12\\
\times\bigg(\int_0^\infty\expec{\Big(\oscd{A_s,B_{2\ell+1}(x)}Z(A_s)\Big)^2}(\ell+1)^{-d}\pi(\ell)\,d\ell\bigg)^\frac12dx.
\end{multline}
We refer to Remark~\ref{rem:cov-alm} for possible reformulation in the canonical form of the multiscale covariance inequality~($\osc$-MCI).
\end{rem}}

\begin{proof}[Proof of Proposition~\ref{prop:genincl-loss}]
We split the proof into two steps. 
We first apply the general results of Theorem~\ref{th:ar}, and then treat more carefully the case of the random parking point process.

\medskip

\step1 Proof of~(i).

\noindent
In the case of a Poisson point process $\Pc=\Pc_0$ on $\R^d$ with constant intensity $\mu>0$,
an action radius for $A_s$ wrt $\X_s$ on $\{x\}\times\{v\}$
 is  given by
\[\rho_{x,v}^s=v\,\mathds1_{\X_s\ne \X_s^{x,v}}.\]
Estimating
\begin{align*}
\pr{\ell-1\le\rho^s_{x,v}< \ell,~\X_s^{x,v}\ne \X_s}
&\le \pr{\X_s^{x,v}\ne \X_s}\mathds1_{\ell-1\le v<\ell}\\
&\le2\mu\gamma(v)\,\mathds1_{\ell-1\le v<\ell},
\end{align*}
and using that $\prm{\rho^s_{x,v}< \ell}=1$ if $v<\ell$, we obtain for all $x\in\Z^d$, $v\ge0$, $\ell\ge1$,
\begin{align*}
\frac{\pr{\ell-1\le \rho^s_{x,v}< \ell,\,\X_s^{x,v}\ne \X}}{\pr{\rho^s_{x,v}< \ell}}&\le
2\mu\gamma(v)\,\mathds1_{\ell-1\le v< \ell},
\end{align*}
so that Theorem~\ref{th:ar} and Remark~\ref{rem:ar-ci} with influence function $f(u)=u$ yield
\begin{multline*}
\cov{Y(A_s)}{Z(A_s)}\\
\le\mu\sum_{x\in\Z^d}\sum_{v=0}^\infty\gamma(v)\,\expec{\Big(\oscd{A_s,Q_{2v+3}(x)}Y(A_s)\Big)^2}^\frac12\expec{\Big(\oscd{A_s,Q_{2v+3}(x)}Z(A_s)\Big)^2}^\frac12.
\end{multline*}
Replacing sums by integrals, the desired covariance estimate~($\osc$-MCI) follows.

\medskip

\step2 Proof of (ii).\\
In this step, we consider the case when the stationary point process $\Pc$ satisfies a hardcore condition $\sharp(\Pc\cap Q)\le C$ a.s.\@ for some deterministic constant $C>0$, and also satisfies the following covariance inequality (resp.\@ the corresponding ($\osc$-MSG)) with some integrable weight $\pi_0$: for all $\sigma(\Pc)$-measurable random variables $Y(\Pc),Z(\Pc)$,
\begin{multline}\label{eq:ass-ci-pc}
\cov{Y(\Pc)}{Z(\Pc)}\le \int_{\R^d}\bigg(\int_0^\infty\expec{\Big(\oscd{\Pc,B_{\ell}(x)}Y(\Pc)\Big)^2}(\ell+1)^{-d}\pi_0(\ell)\,d\ell\bigg)^\frac12\\
\times\bigg(\int_0^\infty\expec{\Big(\oscd{\Pc,B_{\ell}(x)}Z(\Pc)\Big)^2}(\ell+1)^{-d}\pi_0(\ell)\,d\ell\bigg)^\frac12dx.
\end{multline}
We then show that, for each $s=1,2,3$, the random field $A_s$ satisfies the following covariance inequality (resp.\@ the corresponding ($\osc$-MSG)): for all $\sigma(A_s)$-measurable random variables $Y(A_s),Z(A_s)$ we have
{\color{black}\begin{multline}\label{eq:res-RPM-type-cov}
\cov{Y(A_s)}{Z(A_s)}\le \int_{\R^d}\bigg(\int_0^\infty\expec{\Big(\oscd{A_s,B_{2\ell+1}(x)}Y(A_s)\Big)^2}(\ell+1)^{-d}\pi(\ell)\,d\ell\bigg)^\frac12\\
\times\bigg(\int_0^\infty\expec{\Big(\oscd{A_s,B_{2\ell+1}(x)}Z(A_s)\Big)^2}(\ell+1)^{-d}\pi(\ell)\,d\ell\bigg)^\frac12dx,
\end{multline}
where we have set
\[\pi(\ell):=C(\ell+1)^d\Big(\pr{\ell-1\le V<\ell}+\int_0^\ell \pr{r-1\le V<r}\pi_0(\ell-r)\,dr\Big).\]
In particular, combined with Propositions~\ref{prop:rpm}--\ref{prop:SG0}, this implies the covariance inequality~\eqref{eq:MSG-RPM} in the case of the random parking or hardcore Poisson process.}

\smallskip\noindent
To simplify notation, we only treat the case of the Poincaré inequality.
Consider a measurable enumeration of the point process $\Pc=\{Z_j\}_j$, let $\{Z_j,V_j,Y_{s,j}\}$ be a decorated point process associated with $\Pc$ and the decoration law $(V,Y_s)$, and let $\Dc:=\{V_j,Y_{s,j}\}_j$ denote the decoration sequence.
Since $\Pc$ and $\Dc$ are independent, the expectation $\E$ splits into $\E=\E_{\Pc}\E_{\Dc}$, where $\E_{\Pc}=\E[\cdot\|\Dc]$ denotes the expectation wrt $\Pc$, and where $\E_{\Dc}=\E[\cdot\|\Pc]$ denotes the expectation wrt $\Dc$.
{\color{black}By tensorization of the variance in form of
\begin{multline*}
\var{Z(A_s)}=\E_{\Pc}\big[\Var_{\Dc}[Z(A_s)]\big]+\Var_{\Pc}\big[\E_{\Dc}[Z(A_s)]\big]\\
\le \E_{\Pc}\big[\Var_{\Dc}[Z(A_s)]\big]+\E_{\Dc}\big[\Var_{\Pc}[Z(A_s)]\big],
\end{multline*}}
the Poincaré inequality assumption for $\Pc$ (cf.~\eqref{eq:ass-ci-pc}) and the standard Poincaré inequality~\eqref{eq:var-X} for the i.i.d.\@ sequence $\Dc$ then yield for all $\sigma(A_s)$-measurable random variables~$Z(A_s)$,
\begin{multline}\label{eq:var-decomp-RPM}
\var{Z(A_s)}\le \frac12\sum_{k}\expec{\big(Z(A_s)-Z(A_s^k)\big)^2}\\
+\int_0^\infty\int_{\R^d}\expec{\Big(\oscd{\Pc,B_{\ell}(x)}Z(A_s)\Big)^2}dx\,(\ell+1)^{-d}\pi_0(\ell)\,d\ell,
\end{multline}
where $A^k_s$ corresponds to the field $A_s$ with the decoration $(V_k,Y_{s,k})$ replaced by an i.i.d.\@ copy $(V'_k,Y'_{s,k})$.
We separately estimate the two RHS terms in~\eqref{eq:var-decomp-RPM}, and we start with the first.
For all $x\in\R^d$, we define the following two random variables,
{\color{black}\begin{align*}
N(x):=\sharp(\Pc\cap B(x)),\qquad R(x):=\max\{V_j,V_j':Z_j\in B(x)\}.
\end{align*}}
Let $R_0\ge1$ denote the smallest value such that $\pr{V<R_0}\ge\frac12$. By a union bound and the hardcore assumption, there holds
\begin{align}\label{eq:choice-R0-SGRPM}
\pr{R(x)<R_0}=\expec{\pr{V<R_0}^{2N(x)}}\ge \expec{2^{-2N(x)}}\ge 4^{-C}.
\end{align}
Conditioning wrt the value of $R(x)$, we obtain
\begingroup\allowdisplaybreaks
\begin{eqnarray*}
\lefteqn{\sum_{k}\expec{\big(Z(A_s)-Z(A_s^k)\big)^2}}\\
&\lesssim&\int_{R_0}^\infty\int_{\R^d}\sum_{k}\expec{\big(Z(A_s)-Z(A_s^k)\big)^2\mathds1_{Z_k\in B(x)}\mathds1_{\ell-1\le R(x)<\ell}}dx\,d\ell\\
&&\qquad+\int_{\R^d}\sum_{k}\expec{\big(Z(A_s)-Z(A_s^k)\big)^2\mathds1_{Z_k\in B(x)}\mathds1_{R(x)<R_0}}dx\\
&\le&\int_{R_0}^\infty\int_{\R^d}\expec{\Big(\oscd{A_s,B_{\ell+1}(x)}Z(A_s)\Big)^2N(x)\,\mathds1_{\ell-1\le R(x)<\ell}}dx\,d\ell\\
&&\qquad+\int_{\R^d}\expec{\Big(\oscd{A_s,B_{R_0+1}(x)}Z(A_s)\Big)^2N(x)}dx\\
&=&\int_{R_0}^\infty\int_{\R^d}\expeC{\Big(\oscd{A_s,B_{\ell+1}(x)}Z(A_s)\Big)^2N(x)\,\mathds1_{R(x)\ge\ell-1}}{R(x)<\ell}\pr{R(x)<\ell}dx\,d\ell\\
&&\qquad+\int_{\R^d}\expec{\Big(\oscd{A_s,B_{R_0+1}(x)}Z(A_s)\Big)^2N(x)}dx.
\end{eqnarray*}
\endgroup
Using the hardcore assumption in the form $N(x)\le C$ a.s., and noting that given $R(x)<\ell$ the random variable $R(x)$ is independent of $A_s|_{\R^d\setminus B_{\ell+1}(x)}$, we deduce
\begin{align*}
\sum_{k}\expec{\big(Z(A_s)-Z(A_s^k)\big)^2}
&\lesssim\int_{R_0}^\infty\int_{\R^d}\expec{\Big(\oscd{A_s,B_{\ell+1}(x)}Z(A_s)\Big)^2}\frac{\pr{\ell-1\le R(x)<\ell}}{\pr{R(x)<\ell}}dx\,d\ell\\
&\hspace{3cm}+\int_{\R^d}\expec{\Big(\oscd{A_s,B_{R_0+1}(x)}Z(A_s)\Big)^2}dx.
\end{align*}
Estimating by a union bound $\pr{\ell-1\le R(x)<\ell}\le C\,\pr{\ell-1\le V<\ell}$, and making use of the property~\eqref{eq:choice-R0-SGRPM} of the choice of $R_0\ge1$, we conclude
\begin{multline}\label{eq:var-decomp-RPM-1st-term}
{\sum_{k}\expec{\big(Z(A_s)-Z(A_s^k)\big)^2}}\\
\lesssim \int_{0}^\infty\int_{\R^d}\expec{\Big(\oscd{A_s,B_{\ell+1}(x)}Z(A_s)\Big)^2}dx\,\pr{\ell-1\le V<\ell}\,d\ell.
\end{multline}
It remains to estimate the second RHS term in~\eqref{eq:var-decomp-RPM}.
The hardcore assumption for $\Pc$ yields by stationarity $\sharp(\Pc\cap B_{\ell}(x))\le C\ell^d$ a.s.
Also note that a union bound gives
\begin{eqnarray*}
{\pr{r-1\le \max_{1\le j\le C\ell^d}V_j<r}}&\le& \sum_{j=1}^{C\ell^d}\pr{V_j\ge r-1,\text{ and }V_k<r~\forall 1\le k\le C\ell^d}\\
&=& C\ell^d\,\pr{V<r}^{C\ell^d-1}\pr{r-1\le V<r},
\end{eqnarray*}
hence for all $r\ge R_0$,
\begin{align*}
\frac{\pr{r-1\le \max_{1\le j\le C\ell^d}V_j<r}}{\pr{\max_{1\le j\le C\ell^d}V_j<r}}&\le C\ell^d\frac{\pr{r-1\le V<r}}{\pr{V<r}}\le 2C\ell^d\pr{r-1\le V<r}.
\end{align*}
{\color{black}Arguing similarly as above, we then find
\begin{multline*}
{\int_0^\infty\int_{\R^d}\expec{\Big(\oscd{\Pc,B_\ell(x)}Z(A_s)\Big)^2}dx\,(\ell+1)^{-d}\pi_0(\ell)\,d\ell}\\
\lesssim \int_0^\infty\int_{0}^\infty\int_{\R^d}\expec{\Big(\oscd{A_s,B_{\ell+r}(x)}Z(A_s)\Big)^2}dx\,\pr{r-1\le V<r}dr\,\pi_0(\ell)\,d\ell.
\end{multline*}
Combining this with~\eqref{eq:var-decomp-RPM} and~\eqref{eq:var-decomp-RPM-1st-term}, the conclusion~\eqref{eq:res-RPM-type-cov} follows in variance form.}
\end{proof}

\subsection{Dependent coloring of random geometric patterns}\label{chap:dep-col-geom-pat}

Up to here, besides Gaussian random fields, all the examples of random fields that we have been considering corresponded to random geometric patterns (various random point processes constructed from a higher-dimensional Poisson process, or random tessellations) endowed with an independent coloring determining e.g.\@ the size and shape of the cells and the value of the field in the cells. In the present subsection, we 
consider {\it dependent} colorings of random geometric patterns. The random field $A$ is now a function of both a product structure (typically some decorated Poisson point process $\hat\Pc$), and of a random field $G$ (e.g.\@ a Gaussian random field) which typically has long-range correlations but is assumed to satisfy some multiscale functional inequality.
In other words, this amounts to mixing up all the previous examples. Rather than stating general results in this direction, we only treat a number of typical concrete examples in order to illustrate the robustness of the approach.

\begin{itemize}
\item The first example $A_1$ is a random field on $\R^d$ corresponding to random spherical inclusions centered at the points of a Poisson point process $\Pc$ of intensity $\mu=1$, with i.i.d.\@ random radii of law $V$, but such that the values on the inclusions are determined by some random field $G_1$ with long-range correlations.
\\
More precisely, we let $\hat\Pc_1:=\{\tilde P_j,\tilde V_j,\tilde U_j\}_j$ denote a decorated point process associated with $\Pc$ and $(V,U)$, where $U$ denotes an independent uniform random variable on $[0,1]$. Independently of $\hat\Pc_1$ we choose a jointly measurable stationary bounded random field $G_1$ on $\R^d$, with typically long-range correlations. The collection of random inclusions is given by $\{\tilde I_1^j\}_j$ with $\tilde I_1^j:=\tilde P_j+\tilde V_jB$. As in the third example of Section~\ref{chap:genincl}, we choose a $\sigma(V,U)$-measurable random variable $P(V,U)$, and we say that the inclusion $\tilde I_1^j$ has the priority on inclusion $\tilde I_1^i$ if $P(\tilde V_j,\tilde U_j)<P(\tilde V_i,\tilde U_i)$ or if $P(\tilde V_j,\tilde U_j)=P(\tilde V_i,\tilde U_i)$ and $\tilde U_j<\tilde U_i$. This defines a priority order on the inclusions on a set of maximal probability, and we then relabel the inclusions and the points of $\hat\Pc_1$ into a sequence $( I_1^j, P_j,V_j,U_j)_j$ such that for all $j$ the inclusion $I_1^j$ has the $j$-th highest priority. Given $\beta\in\R$, we then consider the $\sigma(\hat\Pc_1,G_1)$-measurable random field $A_1$ defined as follows,
\[A_1:=\beta+\sum_j\big(G_1(P_j)-\beta\big)\mathds1_{I_1^j\setminus\bigcup_{i:i<j}I_1^i}.\]
\item The second example $A_2$ is a random field on $\R^d$ corresponding to random inclusions centered at the points of a Poisson point process $\Pc$ of intensity $\mu=1$, with i.i.d.\@ random radii of law $V$, but with orientations determined by some random field $G_2$ with long-range correlations.
\\
More precisely, we let $\hat\Pc_2:=\{P_j,V_j\}_j$ denote a decorated point process associated with $\Pc$ and $V$, we choose a reference shape $S\in\B(\R^d)$ with $0\in S$, and independently of $\hat\Pc_2$ we choose a jointly measurable stationary bounded random field $G_2$ on $\R^d$ with values in the orthogonal group $O(d)$ in dimension $d$, and with typically long-range correlations. The collection of random inclusions is then given by $\{I_2^j\}_j$ with $I_2^j:=P_j+G_2(P_j) S$. Given $\alpha,\beta\in\R$, and given a function $\phi:\R\to\R$ with $\phi(t)=1$ for $t\le1$ and $\phi(t)=0$ for $t\ge2$, and with $\|\phi'\|_{\Ld^\infty}\lesssim1$, we then consider the $\sigma(\hat\Pc_2,G_2)$-measurable random field $A_2$ defined as follows,
\[A_2(x):=\beta+(\alpha-\beta)\,\phi\Big(d\big(x\,,\,\cup_jI_2^j\big)\Big).\]
(Note that the smoothness of this interpolation $\phi$ between the values $\alpha$ and $\beta$ is crucial for the arguments below.)

\smallskip\item The third example $A_3$ is a random field on $\R^d$ corresponding to the Voronoi tessellation associated with the points of a Poisson point process $\Pc$ of unit intensity, such that the values on the cells are determined by some random field $G_3$ with long-range correlations.
\\
More precisely, we let $\hat\Pc_3:=\Pc=\{P_j\}_j$, and we let $\{C_j\}_j$ denote the partition of $\R^d$ into the Voronoi cells associated with the Poisson points $\{P_j\}_j$. Independently of $\hat\Pc_3$ we choose a jointly measurable stationary bounded random field $G_3$ on $\R^d$.
We then consider the $\sigma(\hat\Pc_3,G_3)$-measurable random field $A_3$ defined as follows,
\[A_3(x):=\sum_jG_3(P_j)\mathds1_{C_j}.\]
\end{itemize}

{\color{black}For each of these examples, we establish functional inequalities with the supremum derivative~$\parsup{}$, cf.~Section~\ref{chap:spectralgaps}.
The proof below is quite robust and many variants could be considered.}

\begin{prop}\label{prop:ex-dep-color}
For $s=1,2,3$, assume that the random field $G_s$ satisfies {\rm($\parfct{}$-MSG)} for some integrable weight $\pi_s$.
For $s=1,2$, set $\gamma(r):=\pr{r-1\le V<r}$.
Then the following holds.
\begin{enumerate}[(i)]
\item
For $s=1,2$, the above-defined random field $A_s$ satisfies the following multiscale Poincaré inequality: for all $\sigma(A_s)$-measurable random variables $Z(A_s)$ we have
\begin{multline}\label{eq:s1-gen}
\qquad \var{Z(A_s)}\\
\qquad \lesssim
\expecM{\int_0^\infty\int_{0}^\infty\int_{\R^d}\Big(\parsup{A,B_{\ell+r+1}(x)}Z(A_s)\Big)^2dx\,\big((\ell+1)^{-d}\wedge\gamma(r)\big)\pi_s(\ell)\,drd\ell}.
\end{multline}
In the case when the random variable $V$ is almost surely bounded by a deterministic constant, we rather obtain
\begin{multline}\label{eq:s1-bound}
\qquad\var{Z(A_s)}\lesssim\expecM{\int_{\R^d}\Big(\oscd{A_s,B_{C}(x)}Z(A_s)\Big)^2dx}\\
+\expecM{\int_0^\infty\int_{\R^d}\Big(\parfct{A_s,B_{\ell+C}(x)}Z(A_s)\Big)^2dx\,(\ell+1)^{-d}\pi_s(\ell)\,d\ell},
\end{multline}
and if the random field $G_s$ further satisfies \emph{($\parfct{}$-MLSI)} with weight $\pi_s$, then the corresponding logarithmic Sobolev inequality also holds, that is,
\begin{multline*}
\qquad\ent{Z(A_s)}\lesssim\expecM{\int_{\R^d}\Big(\oscd{A_s,B_{C}(x)}Z(A_s)\Big)^2dx}\\
+\expecM{\int_0^\infty\int_{\R^d}\Big(\parfct{A_s,B_{\ell+C}(x)}Z(A_s)\Big)^2dx\,(\ell+1)^{-d}\pi_s(\ell)\,d\ell}.
\end{multline*}
\item The above-defined random field $A_3$ satisfies \emph{($\parsup{}$-MSG)}
with weight $\pi(\ell):=C( \pi_3(\ell)+e^{-\frac1C\ell^d})$.
If the random field $G_3$ further satisfies \emph{($\parfct{}$-MLSI)} with weight $\pi_3$, then $A_3$ also satisfies~\emph{($\parsup{}$-MLSI)} with weight $\pi$.
\qedhere
\end{enumerate}
\end{prop}
\begin{proof}
For $s=1,2,3$, since $\hat\Pc_s$ and $G_s$ are independent, the expectation $\E$ splits into $\E=\E_{\hat\Pc_s}\E_{G_s}$, where $\E_{\hat\Pc_s}[\cdot]=\E[\cdot\|G_s]$ denotes the expectation wrt $\hat\Pc_s$, and where $\E_{G_s}[\cdot]=\E[\cdot\|\hat\Pc_s]$ denotes the expectation wrt $G_s$.
{\color{black}The starting point is then the tensorization of the variance and of the entropy,}
\begin{gather}\label{eq:var-tens-exact}
\var{Z(A_s)}=\Var_{G_s}[\E_{\hat\Pc_s}[Z(A_s)]]+\E_{G_s}[\Var_{\hat\Pc_s}[Z(A_s)]],\\
\ent{Z(A_s)}=\Ent_{G_s}[\E_{\hat\Pc_s}[Z(A_s)]]+\E_{G_s}[\Ent_{\hat\Pc_s}[Z(A_s)]].\nonumber
\end{gather}
In each of the examples under consideration, the estimate on the terms $\Var_{\hat\Pc_s}[Z(A_s)]$ and $\Ent_{\hat\Pc_s}[Z(A_s)]$ (with $G_s$ ``frozen'') follows from the same arguments as in the proof of Propositions~\ref{prop:voronoi} and~\ref{prop:genincl-loss}(i). We therefore focus on the estimates of $\Var_{G_s}[\E_{\hat\Pc_s}[Z(A_s)]]$ and $\Ent_{G_s}[\E_{\hat\Pc_s}[Z(A_s)]]$, and only treat the case of the variance in the proof. 

\medskip\noindent
Since the random field $G_s$ is assumed to satisfy ($\parfct{}$-MSG) with weight $\pi_s$, we obtain
\begin{multline}\label{eq:mixex-0}
\Var_{G_s}[\E_{\hat\Pc_s}[Z(A_s)]]\le\E_{\hat\Pc_s}[\Var_{G_s}[Z(A_s)]]\\
\le \expec{\int_0^\infty\int_{\R^d}\Big(\parfct{G_s,B_{\ell+1}(x)}Z(A_s)\Big)^2dx\,(\ell+1)^{-d}\pi_s(\ell)\,d\ell}.
\end{multline}
The chain rule yields
\begin{align*}
\parfct{G_s,B_{\ell+1}(x)}Z(A_s)&=\int_{B_{\ell+1}(x)}\Big|\frac{\partial Z(A_s(\hat\Pc_s,G_s))}{\partial G_s}(y)\Big|dy\\
&\le\int_{B_{\ell+1}(x)}\int_{\R^d}\Big|\frac{\partial Z(A_s)}{\partial A_s}(z)\Big|\Big|\frac{\partial A_s(\hat\Pc_s,G_s)(z)}{\partial G_s}(y)\Big|dzdy.
\end{align*}
Since $A_s$ is $\sigma(\hat\Pc_s,\{G_s(P_j)\}_j)$-measurable, we obtain 
\begin{align}\label{eq:mixex-1}
\parfct{G_s,B_{\ell+1}(x)}Z(A_s)&\le\sum_j\mathds1_{P_j\in B_{\ell+1}(x)}\int_{\R^d}\Big|\frac{\partial Z(A_s)}{\partial A_s}(z)\Big|\Big|\frac{\partial A_s(\hat\Pc_s,G_s)(z)}{\partial G_s(P_j)}\Big|dz
\end{align}
in terms of the usual partial derivative of $A_s(\hat\Pc_s,G_s)(z)$ wrt $G_s(P_j)$.
We now need to compute this derivative in each of the considered examples.
We claim that
\begin{align}\label{eq:partial-der-comp}
\Big|\frac{\partial A_s(\hat\Pc_s,G_s)(z)}{\partial G_s(P_j)}\Big|\le C\mathds1_{R_s^j}(z),
\end{align}
where
\[R_s^j:=\begin{cases}
I_1^j\setminus \bigcup_{i:i<j}I_1^i,&\text{if $s=1$};\\
\big\{x\,:\, 0<d(x,I_2^j)< 2\wedge d(x,I_2^k),\,\forall k\ne j\big\},&\text{if $s=2$};\\
C_j,&\text{if $s=3$}.
\end{cases}\]
This claim~\eqref{eq:partial-der-comp} is obvious for $s=1$ and $s=3$. For $s=2$, the properties of $\phi$ and the definition of $R_2^j$ yield
\[\Big|\frac{\partial A_2(\hat\Pc_2,G_2)(z)}{\partial G_2(P_j)}\Big|\le |\alpha-\beta|\Big|\phi'\Big(d\big(z\,,\,\cup_kI_2^k\big)\Big)\Big|\mathds1_{R_2^j}(z)=|\alpha-\beta|\big|\phi'\big(d(z,I_2^j)\big)\big|\mathds1_{R_2^j}(z),\]
which indeed implies~\eqref{eq:partial-der-comp}. Now injecting~\eqref{eq:partial-der-comp} into~\eqref{eq:mixex-1}, and noting that in each case the sets $\{R_s^j\}_j$ are disjoint, we obtain
\begin{align}\label{eq:mixex-2}
\parfct{G_s,B_{\ell+1}(x)}Z(A_s)&\le C\sum_j\mathds1_{P_j\in B_{\ell+1}(x)}\int_{R_s^j}\Big|\frac{\partial Z(A_s)}{\partial A_s}\Big|=C\int_{\bigcup_{j:P_j\in B_{\ell+1}(x)}R_s^j}\Big|\frac{\partial Z(A_s)}{\partial A_s}\Big|\nonumber\\
&\le C\int_{B_{D_s(\ell,x)}(x)}\Big|\frac{\partial Z(A_s)}{\partial A_s}\Big|,
\end{align}
with
\[D_s(\ell,x):=\sup\Big\{d(y,x)\,:\,y\in \bigcup_{j:P_j\in B_{\ell+1}(x)}R_s^j\Big\}.\]
For $s=1,2$ with radius law $V$ uniformly bounded by a deterministic constant $R>0$, we obtain $D_1(\ell,x)\le\ell+R+1$ and $D_2(\ell,x)\le\ell+R+3$, and injecting~\eqref{eq:mixex-2} into~\eqref{eq:mixex-0} directly yields the result~\eqref{eq:s1-bound}.

\medskip\noindent
We now consider the cases $s=1,2$ with general unbounded radii. Without loss of generality we only treat $s=1$, in which case
\[D_1(\ell,x)\le\ell+1+\bar D_1(\ell,x),\qquad \bar D_1(\ell,x):=\max\big\{V_j\,:\,P_j\in B_{\ell+1}(x)\big\}.\]
Noting that the restriction $A_1|_{\R^d\setminus B_{\ell+1+\bar D_1(\ell,x)}(x)}$ is by construction independent of $\bar D_1(\ell,x)$, we obtain, conditioning on the values of $\bar D_1(\ell,x)$ and arguing as in Step~2 of the proof of Theorem~\ref{th:ar},
\begin{eqnarray}\label{eq:mixex-3}
\lefteqn{\expec{\bigg(\int_{B_{\ell+1+\bar D_1(\ell,x)}(x)}\Big|\frac{\partial Z(A_1)}{\partial A_1}\Big|\bigg)^2}}\nonumber\\
&\le& \int_0^\infty \expeCM{\bigg(\int_{B_{\ell+r+1}(x)}\Big|\frac{\partial Z(A_1)}{\partial A_1}\Big|\bigg)^2\mathds1_{\bar D_1(\ell,x)\ge r-1}}{\bar D_1(\ell,x)<r} \prm{\bar D_1(\ell,x)<r}dr\nonumber\\
&\le& \int_0^\infty \expeCM{\supessd{A_1,B_{\ell+r+1}(x)}\bigg(\int_{B_{\ell+r+1}(x)}\Big|\frac{\partial Z(A_1)}{\partial A_1}\Big|\bigg)^2\mathds1_{\bar D_1(\ell,x)\ge r-1}}{\bar D_1(\ell,x)<r}\nonumber\\
&&\hspace{10cm}\times\prm{\bar D_1(\ell,x)<r}dr\nonumber\\
&\le&\int_0^\infty\expecM{\supessd{A_1,B_{\ell+r+1}(x)}\bigg(\int_{B_{\ell+r+1}(x)}\Big|\frac{\partial Z(A_1)}{\partial A_1}\Big|\bigg)^2}\frac{\prm{r-1\le\bar D_1(\ell,x)<r}}{\prm{\bar D_1(\ell,x)<r}}\,dr.
\end{eqnarray}
Now by definition of the decorated Poisson point process $\hat\Pc_1$, we compute
\begin{align*}
\prm{\bar D_1(\ell,x)\ge r-1}&=\prm{\exists j \,:\, V_j\ge r-1\text{ and }P_j\in B_{\ell+1}(x)}\\
&=e^{-|B_{\ell+1}|}\sum_{n=0}^\infty \frac{|B_{\ell+1}|^n}{n!}\big(1-(1-\pr{V\ge r-1})^n\big)\\
&=1-e^{-|B_{\ell+1}|\,\pr{V\ge r-1}},
\end{align*}
hence
\begin{align*}
\frac{\prm{r-1\le\bar D_1(\ell,x)<r}}{\prm{\bar D_1(\ell,x)<r}}=1-e^{-|B_{\ell+1}|\pr{r-1\le V<r}}\le 1\wedge\big(C(\ell+1)^d\,\pr{r-1\le V<r}\big).
\end{align*}
Combining this computation with~\eqref{eq:mixex-0}, \eqref{eq:mixex-2}, and~\eqref{eq:mixex-3}, we obtain
\begin{multline*}
\Var_{G_1}[\E_{\hat\Pc_1}[Z(A_1)]]
\,\lesssim\, \E\bigg[\int_0^\infty\int_0^\infty\int_{\R^d}\supessd{A_1,B_{\ell+r+1}(x)}\bigg(\int_{B_{\ell+r+1}(x)}\Big|\frac{\partial Z(A_1)}{\partial A_1}\Big|\bigg)^2dx\\
\times\big((\ell+1)^{-d}\wedge\pr{r-1\le V<r}\big)dr\,\pi_s(\ell)\,d\ell\bigg],
\end{multline*}
and the conclusion~\eqref{eq:s1-gen} follows.

\medskip\noindent
We finally turn to the case $s=3$, for which
\[D_3(\ell,x)\le\ell+1+\bar D_3(\ell,x),\qquad \bar D_3(\ell,x):=\max\big\{\diam(C_j)\,:\,P_j\in B_{\ell+1}(x)\big\}.\]
Noting that the restriction $A_3|_{\R^d\setminus B_{\ell+1+2\bar D_3(\ell,x)}(x)}$ is by construction independent of $\bar D_3(\ell,x)$ 
we obtain, after conditioning on the values of $\bar D_3(\ell,x)$ and  arguing as in~\eqref{eq:mixex-3},
\begin{multline}\label{eq:mixex-4}
{\expec{\bigg(\int_{B_{\ell+1+\bar D_3(\ell,x)}(x)}\Big|\frac{\partial Z(A_3)}{\partial A_3}\Big|\bigg)^2}}\,
\le \expec{\supessd{A_3,B_{3\ell+1}(x)}\bigg(\int_{B_{3\ell+1}(x)}\Big|\frac{\partial Z(A_3)}{\partial A_3}\Big|\bigg)^2}
\\+\int_{2\ell}^\infty\expec{\supessd{A_3,B_{\ell+r+1}(x)}\bigg(\int_{B_{\ell+r+1}(x)}\Big|\frac{\partial Z(A_3)}{\partial A_3}\Big|\bigg)^2}\frac{\prm{r-1\le\bar D_3(\ell,x)<r}}{\prm{\bar D_3(\ell,x)<r}}\,dr.
\end{multline}
Similar computations as in Step~2 of the proof of Proposition~\ref{prop:voronoi} yield
\[\prm{\bar D_3(\ell,x)\ge r}\le Ce^{-\frac1C(r-\ell)_+^d}.\]
Combining this with~\eqref{eq:mixex-0}, \eqref{eq:mixex-2}, and~\eqref{eq:mixex-4}, we obtain
\begin{eqnarray*}
\lefteqn{\Var_{G_3}[\E_{\hat\Pc_3}[Z(A_3)]]}\\
&\lesssim& \expecM{\int_0^\infty \int_{\R^d}\supessd{A_3,B_{3\ell+1}(x)}\bigg(\int_{B_{3\ell+1}(x)}\Big|\frac{\partial Z(A_3)}{\partial A_3}\Big|\bigg)^2dx\,(\ell+1)^{-d}\pi_3(\ell)\,d\ell}\\
&&+\expecM{\int_0^\infty\int_{2\ell}^\infty \int_{\R^d}\supessd{A_3,B_{\ell+r+1}(x)}\bigg(\int_{B_{\ell+r+1}(x)}\Big|\frac{\partial Z(A_3)}{\partial A_3}\Big|\bigg)^2dx\,e^{-\frac1Cr^d}\,dr\,(\ell+1)^{-d}\pi_3(\ell)\,d\ell}\\
&\lesssim& \expecM{\int_0^\infty \int_{\R^d}\supessd{A_3,B_{3\ell+1}(x)}\bigg(\int_{B_{3\ell+1}(x)}\Big|\frac{\partial Z(A_3)}{\partial A_3}\Big|\bigg)^2dx\,\big((\ell+1)^{-d}\pi_3(\ell)+e^{-\frac1C\ell^d}\big)\,d\ell},
\end{eqnarray*}
and the result  follows.
\end{proof}


\medskip

\appendix

\addtocontents{toc}{\protect\setcounter{tocdepth}{0}}
\section{Proof of standard functional inequalities}\label{sec:appendA}

In this appendix, we give a proof of Proposition~\ref{prop:criMSG}.

\begin{proof}[Proof of Proposition~\ref{prop:criMSG}]
Let $\e>0$ be fixed, and consider the partition $(Q_x)_{x\in\Z^d}$ of $\R^d$ defined by $Q_x=\e x+\e Q$. Choose an i.i.d.\@ copy $A_0'$ of the field $A_0$, and for all $x$ define the random field $A_0^x$ by $A_0^x|_{\R^d\setminus Q_x}:=A_0|_{\R^d\setminus Q_x}$ and $A_0^x|_{Q_x}:=A_0'|_{Q_x}$. We split the proof into three steps.

\medskip

\step1 Tensorization argument.\\
Choose an enumeration $(x_n)_n$ of $\Z^d$, and for all $n$ let $\Pi_n$ and $\mathbb E_n$ denote the linear maps on $\Ld^2(\Omega)$ defined by
\[\Pi_n:=\expeCm{\cdot}{A_0|_{\bigcup_{k=1}^nQ_{x_k}}},\qquad\qquad\mathbb E_n:=\expeCm{\cdot}{A_0|_{\R^d\setminus Q_{x_n}}}.\]
Also define
\begin{gather*}
\operatorname{Cov}_{n}[Y;Z]:=\mathbb E_n[YZ]-\mathbb E_n[Y]\mathbb E_n[Z],\qquad\operatorname{Var}_{n}[Z]:=\operatorname{Cov}_n[Z;Z],\\
\operatorname{Ent}_n[Z^2]:=\mathbb E_n\big[Z^2\log(Z^2/\mathbb E_n[Z^2])\big].
\end{gather*}
In this step, we make use of a martingale argument \`a la Lu-Yau~\cite{Lu-Yau-93} to show the following tensorization identities for the covariance and for the entropy: for all $\sigma(A_0)$-measurable random variables $Y(A_0),Z(A_0)$, we have
\begin{eqnarray}\label{eq:tens-cov0}
|\cov{Y(A_0)}{Z(A_0)}|&\le&\sum_{k=1}^\infty\expec{\,\big|\operatorname{Cov}_{k}\big[\Pi_{k}[Y(A_0)];\Pi_{k}[Z(A_0)]\big]\big|\,},\\
\ent{Z(A_0)^2}&\le&\sum_{k=1}^\infty\expec{\operatorname{Ent}_{k}\!\big[\Pi_k[Z(A_0)^2]\big]}.\label{eq:tens-ent0}
\end{eqnarray}
First note that for all $\sigma(A_0)$-measurable random variables $Z(A_0)\in\Ld^2(\Omega)$, the properties of conditional expectations ensure that 
$\Pi_n[Z(A_0)]\to Z(A_0)$ in $\Ld^2(\Omega)$ as $n\uparrow\infty$. 
We then decompose the covariance into the following telescopic sum
\begin{eqnarray*}
\lefteqn{\cov{\Pi_n[Y(A_0)]}{\Pi_n[Z(A_0)]}}
\\
&=&\sum_{k=1}^n\Big(\expec{\Pi_k[Y(A_0)]\Pi_k[Z(A_0)]}-\expec{\Pi_{k-1}[Y(A_0)]\Pi_{k-1}[Z(A_0)]}\Big)\\
&=&\sum_{k=1}^n\expec{\operatorname{Cov}_{k}\big[\Pi_{k}[Y(A_0)];\Pi_{k}[Z(A_0)]\big]},
\end{eqnarray*}
so that the result~\eqref{eq:tens-cov0} follows by taking the limit $n\uparrow\infty$.
Likewise, we decompose the entropy into the following telescopic sum
\begin{eqnarray*}
\lefteqn{\ent{\Pi_n[Z(A_0)^2]}}\\
&=&\sum_{k=1}^n\Big(\expec{\Pi_k[Z(A_0)^2]\log(\Pi_k[Z(A_0)^2])}-\expec{\Pi_{k-1}[Z(A_0)^2]\log(\Pi_{k-1}[Z(A_0)^2])}\Big)\\
&=&\sum_{k=1}^n\expec{\operatorname{Ent}_{k}\!\big[\Pi_k[Z(A_0)^2]\big]},
\end{eqnarray*}
and the result~\eqref{eq:tens-ent0} follows in the limit $n\uparrow\infty$.

\medskip

\step2 Preliminary versions of (CI) and (LSI).\\
In this step, we prove that for all $\sigma(A_0)$-measurable random variables $Y(A_0),Z(A_0)$ we have
\begin{eqnarray}
\lefteqn{|\cov{Y(A_0)}{Z(A_0)}\!|}\nonumber
\\
&\le&\frac12\sum_{k=1}^\infty\expec{\big|\Pi_{k}\big[Y(A_0)-Y(A_0^{x_k})\big]\big|\,\big|\Pi_{k}\big[Z(A_0)-Z(A_0^{x_k})\big]\big|}
\nonumber \\
&\le&\frac12\sum_{x\in\Z^d}\expec{\big(Y(A_0)-Y(A_0^{x})\big)^2}^\frac12\expec{\big(Z(A_0)-Z(A_0^{x})\big)^2}^\frac12,\label{eq:tens-cov}
\end{eqnarray}
and 
\begin{gather}
\ent{Z(A_0)}\le2\sum_{x\in\Z^d}\E\bigg[{\supessd{A_0'}\big(Z(A_0)-Z(A_0^{x})\big)^2}\bigg].\label{eq:tens-ent}
\end{gather}
We first prove~\eqref{eq:tens-cov}: we appeal to~\eqref{eq:tens-cov0} in the form
\begin{align*}
|\cov{Y(A_0)}{Z(A_0)}|&\le\frac12\sum_{k=1}^\infty\expec{\big|\mathbb E_k\big[\Pi_{k}[Y(A_0)-Y(A_0^{x_k})]\,\Pi_{k}[Z(A_0)-Z(A_0^{x_k})]\big]\big|}\\
&\le\frac12\sum_{k=1}^\infty\expec{\big|\Pi_{k}[Y(A_0)-Y(A_0^{x_k})]\big|\,\big|\Pi_{k}[Z(A_0)-Z(A_0^{x_k})]\big|},
\end{align*}
which directly yields~\eqref{eq:tens-cov} by Cauchy-Schwarz' inequality.
Likewise, we argue that~\eqref{eq:tens-ent} follows from \eqref{eq:tens-ent0}.
To this aim, we have to reformulate the RHS of \eqref{eq:tens-ent0}: using the inequality $a\log a-a+1\le(a-1)^2$ for all $a\ge0$, we obtain for all $k\ge 0$,
\begin{align*}
\operatorname{Ent}_{k}\big[\Pi_k[Z(A_0)^2]\big]&\le\mathbb E_k[\Pi_k[Z(A_0)^2]]\,\mathbb E_k\bigg[\Big(\frac{\Pi_k[Z(A_0)^2]}{\mathbb E_k[\Pi_k[Z(A_0)^2]]}-1\Big)^2\bigg]\\
&=\frac{\operatorname{Var}_k\big[\Pi_k[Z(A_0)^2]\big]}{\mathbb E_k[\Pi_k[Z(A_0)^2]]}\\
&=\frac{\E_k\big[(\Pi_k[Z(A_0)^2]-\Pi_k[Z(A_0^{x_k})^2])^2\big]}{2\,\E_k[\Pi_k[Z(A_0)^2]]}\\
&=\frac{\E_k\big[(\Pi_k[(Z(A_0)-Z(A_0^{x_k}))(Z(A_0)+Z(A_0^{x_k}))])^2\big]}{2\,\E_k[\Pi_k[Z(A_0)^2]]}\\
&\le\frac{\E_k\big[\Pi_k[(Z(A_0)-Z(A_0^{x_k}))^2]\,\Pi_k[(Z(A_0)+Z(A_0^{x_k}))^2]\big]}{2\,\E_k[\Pi_k[Z(A_0)^2]]}.
\end{align*}
Since  $(A_0,A_0^{x_k})$ and $(A_0^{x_k},A_0)$ have the same law by complete independence, the above implies, using the inequality $(a+b)^2\le2(a^2+b^2)$ for all $a,b\in\R$,
\begin{align*}
\operatorname{Ent}_{k}\big[\Pi_k[Z(A_0)^2]\big]&\le\frac{2\,\E_k\big[\Pi_k[(Z(A_0)-Z(A_0^{x_k}))^2]\,\Pi_k[Z(A_0^{x_k})^2]\big]}{\E_k[\Pi_k[Z(A_0^{x_k})^2]]}\\
&\le2\supessd{A_0'|_{Q_{x_k}}}\Pi_k[(Z(A_0)-Z(A_0^{x_k}))^2]\\
&\le2\,\Pi_k\bigg[\supessd{A_0'|_{Q_{x_k}}}(Z(A_0)-Z(A_0^{x_k}))^2\bigg].
\end{align*}
Estimate~\eqref{eq:tens-ent} now follows from~\eqref{eq:tens-ent0}.

\medskip
\step3 Proof of (CI) and (LSI).\\
We start with the proof of~(CI). 
Since $A=A(A_0)$ is $\sigma(A_0)$-measurable,~\eqref{eq:tens-cov} yields for all $\sigma(A)$-measurable random variables $Y(A),Z(A)$,
\begin{align*}
\big|\cov{Y(A)}{Z(A)}\big|&\le\frac12\sum_{x\in\Z^d}\expec{\big(Y(A)-Y(A(A_0^{x}))\big)^2}^\frac12\expec{\big(Z(A)-Z(A(A_0^{x}))\big)^2}^\frac12.
\end{align*}
Using that $\expeCm{Y(A)}{A_0|_{\R^d\setminus Q_x}}=\expeCm{Y(A(A_0^x))}{A_0|_{\R^d\setminus Q_x}}$ by complete independence of the field $A_0$, 
\[\expec{\big(Y(A)-Y(A(A_0^{x}))\big)^2}=\expec{\Big(\parG{A_0,Q_x}Y(A(A_0))\Big)^2},\]
where we define the Glauber derivative as 
\begin{align*}
\qquad\parG{A,S}Y(A)=\expeCpm{\big(Y(A)-Y(A')\big)^2}{A'|_{\R^d\setminus S}=A|_{\R^d\setminus S}}^\frac12
\end{align*}
 letting $A'$ denote an i.i.d.\@ copy of $A$, and denoting by $\expecp{\cdot}$ the expectation wrt $A'$ only.
Since the conditional expectation $\expeCm{\cdot}{A_0|_{\R^d\setminus Q_x}}$ coincides with the $\Ld^2$-projection onto the $\sigma(A_0|_{\R^d\setminus Q_x})$-measurable functions, and since $\expeCm{Y(A)}{A|_{\R^d\setminus (Q_x+B_R)}}$ is $\sigma(A|_{\R^d\setminus(Q_x+B_R)})$-measurable and therefore $\sigma(A_0|_{\R^d\setminus Q_x})$-measurable by assumption, we have
\[\expec{\Big(\parG{A_0,Q_x}Y(A(A_0))\Big)^2}\le\expec{\Big(\parG{A,Q_x+B_R}Y(A)\Big)^2}.\]
Combining these two observations, we deduce that for all $\sigma(A)$-measurable random variables $Y(A),Z(A)$,
\begin{align*}
\big|\cov{Y(A)}{Z(A)}\big|&\le\frac12\sum_{x\in\Z^d}\expec{\Big(\parG{A,Q_x+B_R}Y(A)\Big)^2}^\frac12\expec{\Big(\parG{A,Q_x+B_R}Z(A)\Big)^2}^\frac12.
\end{align*}
By taking local averages, this turns into
\begin{eqnarray*}
\lefteqn{\big|\cov{Y(A)}{Z(A)}\big|}
\\
&\le&\frac{\e^{-d}}2\sum_{x\in\Z^d}\int_{\e Q}\expec{\Big(\parG{A,y+\e x+\e Q+B_R}Y(A)\Big)^2}^\frac12\expec{\Big(\parG{A,y+\e x+\e Q+B_R}Z(A)\Big)^2}^\frac12dy\\
&=&\frac{\e^{-d}}2\int_{\R^d}\expec{\Big(\parG{A,y+\e Q+B_R}Y(A)\Big)^2}^\frac12\expec{\Big(\parG{A,y+\e z+\e Q+B_R}Z(A)\Big)^2}^\frac12dy\\
&\le&\frac{\e^{-d}}2\int_{\R^d}\expec{\Big(\parG{A,B_{R+\e\sqrt d/2}(y)}Y(A)\Big)^2}^\frac12\expec{\Big(\parG{A,B_{R+\e\sqrt d/2}(y)}Z(A)\Big)^2}^\frac12dy,
\end{eqnarray*}
that is,~(CI) for any radius larger than $R$. 

\smallskip\noindent
We then turn to the proof of~(LSI).
For all $\sigma(A)$-measurable random variables $Z(A)$, the estimate~\eqref{eq:tens-ent} yields
\begin{align*}
\ent{Z(A)}~&\le~2\sum_{x\in\Z^d}\expec{\supessd{A_0'}\big(Z(A(A_0))-Z(A(A_0^{x}))\big)^2}
\\
&\le~ 2\sum_{x\in\Z^d}\expec{\Big(\oscd{A,Q_x+B_R}Z(A)\Big)^2}.
\end{align*}
The desired result~(LSI) then follows from taking local averages.
\end{proof}

\section{Proof for Gaussian fields}\label{chap:condunif}

This section is dedicated to a self-contained proof of Theorem~\ref{cor:gaussiansg}, based on deforming functional inequalities satisfied by i.i.d.\@ Gaussian sequences.
We refer to~\cite[Appendix~A]{DG3} for a more direct proof based on Malliavin technology.
\begin{proof}[Proof of Theorem~\ref{cor:gaussiansg}]
We split the proof into three steps. 
The result follows from a radial change of variables in suitable Brascamp-Lieb inequalities recalled and proved in the first two steps.

\medskip

\step1 Discrete Brascamp-Lieb inequalities: 
Given a standard Gaussian random vector $W:=(W_1,\ldots,W_N)$ with $N$ independent components, and given a linear transformation $F\in\R^{N\times N}$, the transformed random vector $A:=(A_1,\ldots,A_N):=FW$ satisfies for all $\sigma(A)$-measurable random variables $Y(A),Z(A)$,
\begin{eqnarray}
\var{Z(A)}\!\!&\le&\!\! C\sum_{i,j=1}^N|(FF^t)_{ij}|~\expec{\Big|\frac{\partial Z(A)}{\partial A_i}\Big|\Big|\frac{\partial Z(A)}{\partial A_j}\Big|},\label{eq:bl-var-disr}
\\
\ent{Z(A)^2}\!\!&\le&\!\! C\sum_{i,j=1}^N|(FF^t)_{ij}|~\expec{\Big|\frac{\partial Z(A)}{\partial A_i}\Big|\Big|\frac{\partial Z(A)}{\partial A_j}\Big|},\nonumber\\
\cov{Y(A)}{Z(A)}\!\!&\le &\!\!C\sum_{i=1}^N\expec{\Big(\sum_{j=1}^N\frac{\partial{Y(A)}}{\partial A_j}F_{ji}\Big)^2}^\frac12\!\expec{\Big(\sum_{k=1}^N\frac{\partial{Z(A)}}{\partial A_k}F_{ki}\Big)^2}^\frac12.\nonumber
\end{eqnarray}
{\color{black}Starting point is the well-known corresponding inequalities for independent standard Gaussian random variables (cf.~\cite{Gross-75}): for all $\sigma(A)$-measurable random variables $Y(A),Z(A)$,
\begin{eqnarray*}
\var{Z(A)}\!\!&\le&\!\! C\sum_{i=1}^N\expec{\Big(\frac{\partial Z(A)}{\partial W_i}\Big)^2},
\\
\ent{Z(A)^2}\!\!&\le&\!\! C\sum_{i=1}^N\expec{\Big(\frac{\partial Z(A)}{\partial W_i}\Big)^2},
\\
\cov{Y(A)}{Z(A)}\!\!&\le &\!\!C\sum_{i=1}^N\expec{\Big(\frac{\partial{Y(A)}}{\partial W_i}\Big)^2}^\frac12\!\expec{\Big(\frac{\partial{Z(A)}}{\partial W_i}\Big)^2}^\frac12.
\end{eqnarray*}
%
%
%
It remains to examine how those inequalities are deformed under the chain rule when derivatives wrt $W$ are replaced by derivatives wrt $A$. It suffices to estimate
\begin{multline*}
\sum_{i=1}^N\expec{\Big(\frac{\partial{Z(A)}}{\partial W_i}\Big)^2}
= \sum_{i=1}^N\E\bigg[{\Big(\sum_{j=1}^N\frac{\partial{Z(A)}}{\partial A_j}F_{ji} \Big)^2}\bigg]
=\expec{\nabla Z(A)\cdot (FF^t)\nabla Z(A)}\nonumber \\
\le \sum_{i,j=1}^N|(FF^t)_{ij}|~\expec{\Big|\frac{\partial{Z(A)}}{\partial A_i}\Big|\Big|\frac{\partial{Z(A)}}{\partial A_j}\Big|},
\end{multline*}
and the claims follow.
}

\medskip

\step2 Continuum Brascamp-Lieb inequalities: For $A$ as in the statement of Theorem~\ref{cor:gaussiansg},
we have for all $\sigma(A)$-measurable random variables $Y(A),Z(A)$,
\begin{eqnarray}\label{eq:BL-var}
\var {Z(A)}&\le& C\,\expec{\int_{\R^d}\int_{\R^d}\Big|\frac{\partial Z(A)}{\partial A}(z)\Big|\Big|\frac{\partial Z(A)}{\partial A}(z')\Big| |\Cc(z-z')|dzdz'},
\\
\ent{Z(A)^2}&\le& C\,\expec{\int_{\R^d}\int_{\R^d}\Big|\frac{\partial Z(A)}{\partial A}(z)\Big|\Big|\frac{\partial Z(A)}{\partial A}(z')\Big| |\Cc(z-z')|dzdz'},\\
\cov {Y(A)}{Z(A)}&\le& C\int_{\R^d}\expec{\bigg(\int_{\R^d}\Big|\frac{\partial Y(A)}{\partial A}(z)\Big|\,|\F^{-1}(\sqrt{\F\Cc})(x-z)|\,dz\bigg)^2}^\frac12 \nonumber \\
&& \qquad \times\expec{\bigg(\int_{\R^d}\Big|\frac{\partial Z(A)}{\partial A}(z')\Big|\,|\F^{-1}(\sqrt{\F\Cc})(x-z')|\,dz'\bigg)^2}^\frac12 dx.\nonumber 
\nonumber  
\end{eqnarray}
We focus on the Brascamp-Lieb inequality~\eqref{eq:BL-var}.
By an approximation argument, it is enough to establish~\eqref{eq:BL-var} for those random variables $Z(A)$ that depend on $A$ only via their spatial averages on the partition $\{Q_\e(z)\}_{z\in B_R\cap\e\Z^d}$ with $\e,R>0$.
We introduce the following notation for these averages:
\begin{equation}\label{e.51}
A_\e(z):=\fint_{Q_\e(z)}A,\qquad\mbox{for}\;z\in\e\Z^d.
\end{equation}
In this case, the Fréchet derivative $\{\frac{\partial Z}{\partial A}(x)\}_{x\in\mathbb{R}^d}$ and the 
partial derivatives $\{\frac{\partial Z}{\partial A_\e(z)}\}_{z\in\e\Z^d}$ of $Z=Z(A)$ are related via
\begin{equation}\label{e.54}
\e^d\frac{\partial Z}{\partial A}(x)=\frac{\partial Z}{\partial A_\e(z)},\qquad
\mbox{for}\;x\in Q_\e(z),\,z\in \e\Z^d.
\end{equation}
We infer from (\ref{e.51}) that $\{A_\e(z)\}_{z\in\e\Z^d}$ is a discrete centered Gaussian random field (which is now stationary wrt the action of $\e\Z^d$),
characterized by its covariance
\begin{equation}\label{e.52}
\mathcal C_\e(z-z'):=\fint_{Q_\e(z)}\fint_{Q_\e(z')}\mathcal C(x-x')dx' dx.
\end{equation}
By the discrete result~\eqref{eq:bl-var-disr} in Step~1, we deduce for all $\e,R>0$ and all random variables $Z(A)$ that depend on $A$ only via its spatial averages on the partition $\{Q_\e(z)\}_{z\in B_R\cap\e\Z^d}$,
\begin{gather*}\label{e.53}
\var{Z(A)}\le C\sum_{z\in B_R\cap\e\Z^d}\sum_{z'\in B_R\cap\e\Z^d}|\Cc_\e(z-z')|\,\expec{\Big|\frac{\partial Z}{\partial A_\e(z)}\Big|\Big|\frac{\partial Z}{\partial A_\e(z')}\Big|}.
\end{gather*}
Injecting~\eqref{e.54} and~\eqref{e.52}, the conclusion~\eqref{eq:BL-var} follows.

\medskip

\step3 Conclusion.\\
We focus on the proof for the variance and the covariance (the arguments for the entropy are similar).
By the Brascamp-Lieb inequality \eqref{eq:BL-var}, if $x\mapsto \sup_{B(x)}|\Cc|$ is integrable, the inequality $|ab|\le (a^2+b^2)/2$ for $a,b\in\R$ directly yields for all $\sigma(A)$-measurable random variables $Z(A)$ and all $R>0$ (after taking local averages),
\begin{align*}
\var{Z(A)}&\le C\,\expec{\int_{\R^d}\int_{\R^d} \Big|\frac{\partial Z(A)}{\partial A}(x)\Big|\Big|\frac{\partial Z(A)}{\partial A}(x')\Big||\Cc(x-x')|dxdx'}
\\
&\le2C\,\Big\|\sup_{B_{2R}(\cdot)}|\Cc|\Big\|_{\Ld^1}\expec{\int_{\R^d} \bigg(\fint_{B_R(x)}\Big|\frac{\partial Z(A)}{\partial A}\Big|\bigg)^2dx}.
\end{align*}
Now assume that the covariance function $\Cc$ is not integrable, and that $\sup_{B(x)}|\Cc|\le c(|x|)$ for some Lipschitz function $c:\R_+\to\R_+$. Given a $\sigma(A)$-measurable random variable $Z(A)$, we consider the projection $Z_R(A):=\E[Z(A)\| A|_{B_R}]$, for $R>0$. Taking local averages, using polar coordinates, and integrating by parts (note that there is no boundary term since the Fréchet derivative $\partial Z_R(A)/\partial A$ is compactly supported in $B_R$), the Brascamp-Lieb inequality~\eqref{eq:BL-var} yields
\begin{eqnarray*}
\lefteqn{\var{Z_R(A)}}
\\
&\le& C\,\E\bigg[{\int_{\R^d}\int_{\Sp^{d-1}}\int_0^\infty 
\Big|\frac{\partial Z_R(A)}{\partial A}(x) \Big|\fint_{B(x+\ell u)}\Big|\frac{\partial Z_R(A)}{\partial A}(u')\Big|du' \ell^{d-1}c(\ell)d\ell d\sigma(u)dx}\bigg]\\
&=&C\,\E\bigg[{\int_{\R^d}\Big|\frac{\partial Z_R(A)}{\partial A}(x) \Big|\int_{\Sp^{d-1}}\int_0^\infty\int_0^\ell
\fint_{B(x+s u)}\Big|\frac{\partial Z_R(A)}{\partial A}(u')\Big|du' s^{d-1}ds(-c'(\ell))d\ell d\sigma(u)dx}\bigg]\\
&\le&C\,\E\bigg[{\int_{\R^d} \Big|\frac{\partial Z_R(A)}{\partial A}(x) \Big|\int_0^\infty\bigg(\int_{B_{\ell+1}(x)}\Big|\frac{\partial Z_R(A)}{\partial A}\Big|\bigg)(-c'(\ell))d\ell dx}\bigg].
\end{eqnarray*}
Reorganizing the integrals, and taking local spatial averages, we conclude
\begin{eqnarray*}
\lefteqn{\var{Z_R(A)}}
\\
&\le &C\,\expec{\int_0^\infty \int_{\R^d} \Big|\frac{\partial Z_R(A)}{\partial A}(x) \Big|\Big(\parfct{A,B_{\ell+1}(x)}Z_R(A)\Big) dx(-c'(\ell))_+d\ell}\\
&\le&C\,\expec{\int_0^\infty \int_{\R^d}\int_{B_{\ell+1}}\Big|\frac{\partial Z_R}{\partial A}(x+y)\Big|\Big(\parfct{A,B_{\ell+1}(x+y)}Z_R(A)\Big)dydx\,(\ell+1)^{-d}(-c'(\ell))_+d\ell}\\
&\le&C\,\expec{\int_0^\infty \int_{\R^d}\Big(\parfct{A,B_{2(\ell+1)}(x)}Z_R(A)\Big)^2dx\,(\ell+1)^{-d}(-c'(\ell))_+d\ell}\\
&\le &C\, \expec{\int_0^\infty \int_{\R^d}\Big(\parfct{A,B_{\ell+1}(x)}Z_R(A)\Big)^2dx\,(\ell+1)^{-d}(-c'(\ell))_+d\ell},
\end{eqnarray*}
where in the last line we used the (sub)additivity of $S\mapsto\parfct{A,S}$.
By Jensen's inequality in the form
\[\expec{\Big(\parfct{A,S}Z_R(A)\Big)^2}\le\expec{\Big(\expeC{\parfct{A,S}Z(A)}{A|_{B_R}}\Big)^2}\le\expec{\Big(\parfct{A,S}Z(A)\Big)^2},\]
and passing to the limit $R\uparrow\infty$, the conclusion~($\parfct{}$-MSG) follows.
Let us now turn to the covariance inequality. Assuming that $\sup_{B(x)}|\F^{-1}(\sqrt{\F\Cc})|\le r(|x|)$ for some Lipschitz function $r:\R_+\to\R_+$, a radial integration by parts similar as above yields
\begin{multline*}
\cov {Y_R(A)}{Z_R(A)}\lesssim \int_{\R^d}\expecM{\bigg(\int_0^\infty \Big(\parfct{A,B_{\ell+1}(x)}Y_R(A)\Big)(-r'(\ell))_+\,d\ell\bigg)^2}^\frac12\\
\times\expecM{\bigg(\int_0^\infty \Big(\parfct{A,B_{\ell'+1}(x)}Z_R(A)\Big)(-r'(\ell'))_+\,d\ell'\bigg)^2}^\frac12 dx.
\end{multline*}
By the triangle inequality, this turns into
\begin{align*}
\cov {Y_R(A)}{Z_R(A)}&\lesssim \int_0^\infty\int_0^\infty\int_{\R^d}\expecM{\Big(\parfct{A,B_{\ell+1}(x)}Y_R(A)\Big)^2}^\frac12\\
&\hspace{2cm}\times\expecM{\Big(\parfct{A,B_{\ell'+1}(x)}Z_R(A)\Big)^2}^\frac12 dx\,(-r'(\ell))_+\,d\ell(-r'(\ell'))_+\,d\ell'\\
&\le2\int_0^\infty\int_{\R^d}\expecM{\Big(\parfct{A,B_{\ell+1}(x)}Y_R(A)\Big)^2}^\frac12\expecM{\Big(\parfct{A,B_{\ell+1}(x)}Z_R(A)\Big)^2}^\frac12dx\\
&\hspace{4cm}\times \Big(\int_0^\ell(-r'(\ell'))_+\,d\ell'\Big)(-r'(\ell))_+\,d\ell,
\end{align*}
and the conclusion~($\parfct{}$-MCI) follows after passing to the limit $R\uparrow\infty$.
\end{proof}

\begin{rem}\label{rem:rescale-2}
We address the claim of Remark~\ref{rem:proofplus} in the context of Gaussian random fields. 
By definition, for all $L\ge1$, the rescaled field $A_L:=A(L\cdot)$ has covariance $\Cc_L:=\Cc(L\cdot)$ and for $|x|\ge1$ it satisfies $\sup_{B(x)}|\Cc_L|=\sup_{B_L(Lx)}|\Cc|\le c((L|x|-L+1)_+)\le c(|x|)$ since $c$ is non-increasing.
This shows that the same conclusions as for $A$ in Theorem~\ref{cor:gaussiansg} also hold for $A_L$ uniformly wrt $L\ge1$.
\end{rem}


\section*{Acknowledgements}
The work of MD is supported by F.R.S.-FNRS (Belgian National Fund for Scientific Research) through a Research Fellowship.
The authors acknowledge financial support from the European Research Council under
the European Community's Seventh Framework Programme (FP7/2014-2019 Grant Agreement
QUANTHOM 335410).

\addtocontents{toc}{\protect\setcounter{tocdepth}{1}}
\bigskip
\bibliographystyle{plain}

\def\cprime{$'$} \def\cprime{$'$}

\end{document}